\documentclass[a4paper]{article}
\usepackage[utf8]{inputenc}
\usepackage{amssymb,amsmath,amsfonts,amsthm}
\usepackage{anysize, tikz, caption}
\usetikzlibrary{arrows}
\marginsize{1in}{1in}{1in}{1in}
\usepackage{graphicx}
\usepackage{xcolor, comment}
\usepackage{enumerate,tabularray}
\usepackage[font=footnotesize]{caption}
\date{}
\usetikzlibrary{arrows.meta}
\usepackage{multicol}
\usepackage{tabularray}

\usepackage{hyperref}
\hypersetup{colorlinks=true,linkcolor=blue}

\newtheorem{theorem}{Theorem}[section]
\newtheorem{lemma}[theorem]{Lemma}

\newtheorem{proposition}[theorem]{Proposition}
\newtheorem{conjecture}{Conjecture}

\newtheorem{remark}[theorem]{Remark}
\newtheorem{observation}[theorem]{Observation}

\DeclareMathOperator{\HIV}{\mathbf{HIV}}
\DeclareMathOperator{\hiv}{\mathbf{hiv}}
\DeclareMathOperator{\HH}{H}
\DeclareMathOperator{\II}{I}
\DeclareMathOperator{\DD}{D}
\DeclareMathOperator{\CC}{C}
\DeclareMathOperator{\Ha}{\widehat{H}}

\usepackage{authblk}

\title{Extinction thresholds in a graph-based model of HIV infection dynamics}

\author[1]{Manuel A. Espinosa-García \thanks{manuel.espinosa@.cua.uam.mx}}
\author[2]{Ana Paulina Figueroa \thanks{apaulinafg@gmail.com}}
\author[3]{Julián A. Fresán-Figueroa \thanks{jfresan@correo.cua.uam.mx}}
\author[4]{Gerardo L. Maldonado \thanks{gmaldonado@im.unam.mx}}
\author[5]{L. Ariadna Sánchez-Solís \thanks{lizzeth.sanchez@cua.uam.mx}}
\affil[1,3,5]{Departamento de Matemáticas Aplicadas y Sistemas, UAM-Cuajimalpa}
\affil[2]{Departamento Acad\'emico de Matem\'aticas, ITAM}
\affil[4]{Instituto de Matemáticas, UNAM}

\begin{document}

\maketitle
   
\begin{abstract}
We study a graph-based cellular automaton for HIV infection dynamics in lymph-node networks, originally introduced by Mukwembi. Each vertex represents a cell site that may be healthy, infected, or dead, and the evolution is controlled by a replacement parameter $R$, which determines whether a dead cell is replaced by an infected or a healthy cell according to the number of its infected neighbors. For a graph $G$, we introduce two extinction parameters. The parameter $\operatorname{hiv}(G)$ is the smallest value of $R$ for which extinction occurs for every admissible initial configuration, whereas $\operatorname{HIV}(G)$ is the smallest threshold such that extinction occurs for every replacement parameter greater than or equal to it. We prove the general bounds $2\leq \operatorname{hiv}(G)\leq \operatorname{HIV}(G)\leq \Delta(G)+1$ and characterize the extremal case $\operatorname{HIV}(G)=\Delta(G)+1$. We also show that the gap $\operatorname{HIV}(G)-\operatorname{hiv}(G)$ is unbounded and determine both parameters for some classical families of graphs. Finally, we study the dynamics of the model using the state-transition digraph of the system and the configurations whose trajectories converge to nontrivial periodic orbits. The results show that extinction depends not only on the replacement parameter but also on the structural properties of the underlying graph.
\end{abstract}

\section{Introduction}

The study of epidemic processes on graphs has attracted considerable attention in recent decades. In these models, vertices represent individuals, cells, or sites, while edges describe possible transmission or local interaction. One of the central questions is to determine the conditions under which an infection persists or eventually disappears from the system. Previous research has shown that the combinatorial structure of the graph plays a fundamental role in the existence of epidemic thresholds and in the long-term behavior of the process \cite{PastorSatorras2001, VanMieghem2009, PastorSatorras2015}. This interaction between network structure and dynamical behavior has motivated the development of graph-based models for a wide variety of biological and epidemiological systems.

The spread of HIV in lymphatic tissues has been studied using models where local interactions between cells determine the global behavior of the infection. Experimental evidence shows that most HIV replication takes place in lymph nodes, which can be viewed as networks of sites where cells interact \cite{Cohen1997}. This motivates the use of graph-theoretic models to describe both the structure of the tissue and the dynamics of the infection.

In \cite{mukwembi08}, Mukwembi proposed a graph-based model for HIV infection dynamics in lymph nodes. In this model, the lymph node is represented by a simple graph $G$, where each vertex represents a site occupied by a cell and edges represent adjacency between sites. Each vertex is assigned one of three possible states, represented by a function $f:V(G)\to\{0,1,2\}$. The value $0$ represents a healthy cell, $1$ an infected cell, and $2$ a dead cell. At the initial time, only healthy and infected vertices are considered, since no cell is assumed to be dead at the beginning of the process.

Given a positive integer $R$, called the replacement parameter, the model evolves in discrete time steps according to the following rules:

\begin{enumerate}
    \item A healthy cell becomes infected if it has at least one infected neighbor; otherwise, it remains healthy.
    \item An infected cell dies in the next time step.
    \item A dead cell is replaced by an infected cell if it has at least $R$ infected neighbors; otherwise, it is replaced by a healthy cell.
\end{enumerate}

These rules describe the interaction between infection, cell death, and replacement. Globally, an initial state is represented by a function $f_0:V(G)\to\{0,1,2\}$ such that $f_0^{-1}(2)=\emptyset$, and the infection evolves according to the previous rules. This evolution is represented by a state sequence $(f_0,f_1,f_2,\ldots),$ where $f_t$ describes the state of the system at time $t$.

For an initial state $f_0$ on a graph $G$ and a replacement parameter $R$, we are interested in determining whether there exists a time $t_0$ such that every vertex is healthy from time $t_0$ onward. If this happens, we say that the infection becomes extinct on $(G,f_0)$ under the replacement parameter $R$; otherwise, we say that the infection is is latent on $(G,f_0)$ for $R$.

Mukwembi proved that, for every admissible initial state on a graph $G$ with $n$ vertices, the infection becomes extinct when $R=n-1$ \cite{mukwembi08}. This result suggests the existence of a threshold between extinction and latency. However, the extinction property is not obviously monotone in $R$. Increasing $R$ makes the reinfection of dead vertices harder, but the global dynamics is synchronous and different values of $R$ may produce different transient configurations. Therefore, for a fixed pair $(G,f_0)$, it is useful to distinguish between the first value of $R$ that leads to extinction and the threshold after which extinction occurs for every larger value of $R$.

Motivated by this distinction, we introduce two extinction parameters. The parameter $\hiv(G)$ is the first replacement parameter that guarantees the extinction of the infection for every admissible initial state of $G$. The parameter $\HIV(G)$ is the eventual extinction threshold, that is, the first value $R_0$ such that every $R\geq R_0$ leads to extinction of the infection for every admissible initial state of $G$. Equivalently, $\HIV(G)$ is one more than the largest replacement parameter for which latency is still possible. This distinction is necessary because, in this synchronous cellular automaton, extinction for one value of $R$ does not automatically imply extinction for all larger values.

Graph-based epidemic models are closely related to cellular automata, where the state of each vertex evolves according to local interaction rules involving its neighbors. Cellular automata have been widely used to model infectious diseases and other biological processes because they provide a natural framework for describing how complex global behaviors emerge from simple local interactions \cite{WhiteDelReySanchez2007,Pfeifer2008}. These models can also be viewed as finite dynamical systems, in which a global transition function acts on the finite set of all possible configurations of the network. This viewpoint allows the use of dynamical tools such as phase spaces, periodic orbits, fixed points, and attractors to study long-term behavior \cite{JarrahLaubenbacher2008,GarciaJarrahLaubenbacher2001}. The graph-based HIV model introduced by Mukwembi fits naturally into this framework, since its evolution is completely determined by local update rules and a replacement parameter governing the reinfection process.

The content of this paper is organized as follows. In Section \ref{sec:extinctionT} we formalize extinction and latency through extinction and latency sets, which leads to the parameters $\hiv(G)$ and $\HIV(G)$. Then we prove the general bounds $2\leq \hiv(G)\leq \HIV(G)\leq \Delta(G)+1$ for every connected graph with at least one edge and characterize the extremal case $\HIV(G)=\Delta(G)+1$. Finally, we show that the gap $\HIV(G)-\hiv(G)$ is unbounded. In Section \ref{sec:rState}, we introduce the $R$-state transition digraph of the model, which gives a phase-space interpretation of extinction and latency in terms of convergence. Here we stablish some useful properties about connected components of this transition digraph. Also, we show that for any positive integer $n$ there exists a graph and a replacement parameter such that this digraph has a cycle of length $n$. Moreover, we characterize settings where there exist cycles of length $2$. In Section \ref{sec:graphFamilies} we determine $\hiv(G)$ and $\HIV(G)$ on $C_n$, $K_n$ and $K_{n,m}$. Also, we show computational results on wheel graphs $W_n$. In Section \ref{sec:conclusions} we comment some conclusions and propose directions for future research.

Throughout the paper, all graphs are finite, simple, connected, and have at least three vertices, unless otherwise stated. For a vertex $v\in V(G)$, we denote by $N(v)$ the set of neighbors of $v$ and by $d(v)=|N(v)|$ its degree. Also, we denote by
$\Delta(G)=\max\{d(v):v\in V(G)\}$ and $\delta(G)=\min\{d(v):v\in V(G)\}$ the maximum and minimum degree of $G$, respectively. If the graph is clear from context, we simply write $\Delta$ and $\delta$.

\section{Extinction thresholds in the graph-based HIV model} \label{sec:extinctionT}

Given a graph $G$, a state of the system is a function $f:V(G)\longrightarrow \{0,1,2\}$, where $0$, $1$, and $2$ represent a healthy, infected, and dead cell, respectively. An \emph{admissible initial state} is a function $f_0:V(G)\longrightarrow \{0,1\}$, equivalently, $f_0^{-1}(2)=\emptyset$. We denote the set of admissible initial states of $G$ by $\Omega_{0}(G)=\{f:V(G)\to\{0,1\}\}.$

Let $R$ be a positive integer, the replacement parameter. Given a state $f_t$ at time $t$, for each $v\in V(G)$ let $d_{t,\II}(v)=|\{x\in N(v):f_t(x)=1\}|$. The state $f_{t+1}$ is determined by the following rules:
\begin{equation}\label{eq:rules}
f_{t+1}(v)=\begin{cases}
0, & \text{if } f_t(v)=0 \text{ and } d_{t,\II}(v)=0,\\
0, & \text{if } f_t(v)=2 \text{ and } d_{t,\II}(v)<R,\\
1, & \text{if } f_t(v)=0 \text{ and } d_{t,\II}(v)\geq 1,\\
1, & \text{if } f_t(v)=2 \text{ and } d_{t,\II}(v)\geq R,\\
2, & \text{if } f_t(v)=1.
\end{cases}
\end{equation}

For $f_0\in\Omega_0(G)$ and $R\in\mathbb{Z}^{+}$, let $(f_t^{R,f_0})_{t\geq 0}$ denote the sequence generated by \eqref{eq:rules} from $f_0$. When no confusion can arise, we simply write $f_t$. Let $\boldsymbol{0}_G\in\Omega_0(G)$ denote the all-healthy state, that is, $\boldsymbol{0}_G(v)=0$ for every $v\in V(G).$ We shall use the following result of Mukwembi.

\begin{theorem}[Mukwembi, \cite{mukwembi08}]\label{Teo:n-1}
Let $G$ be a graph of order $n\geq 3$, and let $f_0\in\Omega_0(G)$. If $R=n-1$, then there exists $t_0\geq 0$ such that $f_{t_0}^{R,f_0}=\boldsymbol{0}_G$.
\end{theorem}

Since $\boldsymbol{0}_G$ is a fixed point of the dynamics, reaching $\boldsymbol{0}_G$ at some time is equivalent to remaining in the all-healthy state at every subsequent time. Let $G$ be a graph and let $f_0\in\Omega_0(G)$. We define the \emph{extinction set} of $(G,f_0)$ as  the set $ \mathcal{E}(G,f_0) = \left\{ R\in\mathbb{Z}^{+}: f_t^{R,f_0}=\boldsymbol{0}_G \text{ for some }t\geq 0 \right\}$. Its complement $ \mathcal{L}(G,f_0) = \mathbb{Z}^{+}\setminus\mathcal{E}(G,f_0) $ is called the \emph{latency set} of $(G,f_0)$. For a graph $G$, we define $ \mathcal{E}(G) = \bigcap_{f_0\in\Omega_0(G)} \mathcal{E}(G,f_0) $ and $ \mathcal{L}(G) = \mathbb{Z}^{+}\setminus\mathcal{E}(G) = \bigcup_{f_0\in\Omega_0(G)} \mathcal{L}(G,f_0)$.

Thus, $R\in\mathcal{E}(G)$ if and only if the infection becomes extinct for every admissible initial state of $G$, whereas $R\in\mathcal{L}(G)$ if and only if there exists an admissible initial state for which the infection is latent.

For $f_0\in\Omega_0(G)$, define the \emph{first extinction value} of $(G,f_0)$ by $\hiv(G,f_0)=\min\mathcal{E}(G,f_0)$, and its \emph{eventual extinction threshold} by $\HIV(G,f_0)=\min\left\{R_0\in\mathbb{Z}^{+}:\{R\in\mathbb{Z}^{+}:R\geq R_0\}\subseteq\mathcal{E}(G,f_0)\right\}$. Equivalently, $\HIV(G,f_0)=1+\max\mathcal{L}(G,f_0)$, with the convention that $\max\emptyset=0$.

Similarly, define $\hiv(G)=\min\mathcal{E}(G)$ and $\HIV(G)=\min\left\{R_0\in\mathbb{Z}^{+}:\{R\in\mathbb{Z}^{+}:R\geq R_0\}\subseteq\mathcal{E}(G)\right\}$. Equivalently, $\HIV(G)=1+\max\mathcal{L}(G)$.

Thus, $\hiv(G)$ is the first replacement parameter that guarantees the extinction of the infection for every admissible initial state, whereas $\HIV(G)$ is one more than the last replacement parameter for which latency is still possible. Theorem \ref{Teo:n-1}, together with the fact that no dead vertex can be reinfected when $R\geq n$, implies that each extinction set contains all sufficiently large integers. Hence, both $\hiv(G)$ and $\HIV(G)$  are well defined.

\begin{observation}\label{obs:minmax}
For every admissible initial state $f_0\in\Omega_{0}(G)$, $ \hiv(G,f_0)\leq \HIV(G,f_0)$, and, at the graph level, $\hiv(G)\leq \HIV(G)$. Moreover, since $\Omega_0(G)$ is finite, $\HIV(G)=\max_{f_0\in\Omega_0(G)}\HIV(G,f_0).$

On the other hand, in general one only has $ \max_{f_0\in\Omega_0(G)} \hiv(G,f_0) \leq \hiv(G),$ and the inequality may be strict because the first extinction value may depend on the initial state.
\end{observation}

For a fixed trajectory in the system, let $\HH_t=\{v\in V(G):f_t(v)=0\},$ $\II_t=\{v\in V(G):f_t(v)=1\}$ and $ \DD_t=\{v\in V(G):f_t(v)=2\}$ be the sets of healthy, infected, and dead vertices at time $t$, respectively. We distinguish between two types of healthy vertices. Let $\Ha_t = \{v\in\HH_t:f_{t'}(v)=0\text{ for every }0\leq t'\leq t\}$ be the set of vertices that have never been infected up to time $t$, and let $\CC_t = \{v\in\HH_t:f_{t'}(v)=1\text{ for some }0\leq t'<t\}$ be the set of cured vertices at time $t$ (the healthy vertices that have been infected at some time $t'<t$). Therefore, $V(G)=\Ha_t\cup\II_t\cup\DD_t\cup\CC_t$ is a partition of $V(G)$ for every $t\geq 0$.

If $X,Y\in\{\Ha,\II,\DD,\CC\}$, we say that an edge $uv$ is of type $(XY)_t$ if one endpoint belongs to $X_t$ and the other belongs to $Y_t$.

\begin{lemma}\label{Lemma:Estructura}
Let $G$ be a graph and let $R\geq\Delta(G)$. Then, for every admissible initial state and every time $t\geq 0$, there are no edges of type $(\Ha\DD)_t$, $(\Ha\CC)_t$, or $(\II\CC)_t$.
\end{lemma}

\begin{proof}
Since $\DD_0=\CC_0=\emptyset$, none of the edge types $(\Ha\DD)_0$, $(\Ha\CC)_0$, or $(\II\CC)_0$ occurs at time $0$.

Suppose first that, for some $t>0$, there is an edge $uv$ of type $(\Ha\DD)_t$ or $(\Ha\CC)_{t}$, with $u\in \Ha_{t}$. Since $v\in \DD_{t}$ or $v\in \CC_{t}$, there exists $t'<t$ for which $v\in \II_{t'}$ and $u\in \Ha_{t'}$. Thus, $u$ has an infected neighbor at time $t'$, namely $v$, and consequently, $u\in \II_{t'+1}$. Since $t'+1\le t$, it contradicts the fact that $u\in \Ha_{t}$.

Now, suppose that $t>0$ is the first time at which an edge $uv$ of type $(\II\CC)_t$ occurs, with $u\in\II_t$ and $v\in\CC_t$. By \eqref{eq:rules}, $u\in\Ha_{t-1}\cup\DD_{t-1}\cup\CC_{t-1}$ and $v\in\DD_{t-1}\cup\CC_{t-1}.$ If $u\in \Ha_{t-1}$, this implies that $uv$ is an edge of type $(\Ha\DD)_{t-1}$ or $(\Ha\CC)_{t-1}$, but we already proved the existence of these type of edges cannot occur. If $u\in \CC_{t-1}$, then $u$ can become infected at time $t$ only if it has an infected neighbor at time $t-1$, which would produce an edge of type $(\II\CC)_{t-1}$, contradicting the minimality of $t$. Finally, we consider the case when $u\in \DD_{t-1}$. Since $v$ is not infected at time $t-1$ and $uv\in E(G)$, the vertex $u$ has at most $d(u)-1\le \Delta(G)-1<R $ infected neighbors at time $t-1$. Therefore $u$ becomes healthy, rather than infected, at time $t$, contradicting $u\in \II_{t}$. Hence no edge of type $(\II\CC)_{t}$ can occur.
\end{proof}

In Figure \ref{fig:aristasmalas} we have a graph with $\Delta=4$. When $R=\Delta=4$, every edge satisfies the conclusion of Lemma \ref{Lemma:Estructura}. The state sequence when $R=2$ contains the edge $v_{0}v_{4}$ of type $(\II\CC)_{3}$.

\begin{observation}\label{obs:cured}
Let $G$ be a graph and let $R\geq\Delta(G)$. If $v\in\CC_{t_0}$, then $v\in\CC_t$ for every $t\geq t_0.$ Indeed, by Lemma \ref{Lemma:Estructura}, a cured vertex has no infected neighbors and therefore remains healthy. Furthermore, if $R\geq\Delta(G)+1$ and $v\in\II_t$, then $ v\in\DD_{t+1} $ and $v\in\CC_{t+2}.$
\end{observation}

The previous result gives a useful interpretation of the dynamics. Under $R\geq\Delta(G)+1$, each vertex can move only through the sequence 
\begin{equation}\label{eq:wellMoveDiagram}
    \Ha\longrightarrow\II \longrightarrow\DD\longrightarrow\CC,    
\end{equation}
possibly remaining in $\Ha$ forever, and no vertex can move backwards. Moreover, for any $R\in\mathbb{Z}^+$, this dynamics can only be broken by a vertex in $\DD_t\cup\II_{t+1}$. 

\begin{proposition}\label{prop:reinfection}
Let $G$ be a finite graph, let $f_0\in\Omega_0(G)$, and let $R\in\mathbb{Z}^{+}$. If $R\in\mathcal{L}(G,f_0)$, then there exist a vertex $v$ and a time $t$ such that $v\in \DD_{t}\cap \II_{t+1}$.
\end{proposition}

\begin{proof}
    Suppose that $\DD_t\cap \II_{t+1}=\varnothing$ for every $t\geq0$. Then no dead vertex is ever reinfected. Consequently, for any vertex $v$ its neighbors are contained in the same type set or in an adjacent one in Diagram \ref{eq:wellMoveDiagram}. Therefore, vertices in $\CC$ can not be infected and $R\not\in\mathcal{L}(G,f_0)$.
\end{proof}

Although $\DD_{t}\cap \II_{t+1}=\emptyset$ for every $t$ implies $R\in\mathcal{L}(G,f_0)$, the converse is not true. As illustrated in Figure \ref{fig:aristasmalas}, in both $R=2$ and $R=4$, there exists at least one vertex that belongs to $\DD_{t}\cap \II_{t+1}$ for some $t$. In the case $R=2$, the infection is latent (it is not difficult to notice that the vertices $v_{0}$, $v_{1}$, $v_{3}$ and $v_{5}$ will be switching between the infected and dead states for every $t\ge 2$). For $R=4$, even when $v_3\in \DD_{1}\cap \II_{2}$, the infection becomes extinct.

\begin{center}
\begin{figure}[ht]
\begin{tikzpicture}[scale=0.74, transform shape,every node/.style={inner sep=1pt, font=\small}]
\draw (19.5,1.5) node {$\mathbf{R=2}$};
\begin{scope}[shift={(0,0)}]
    \draw (0.5,4)--(0.5,-2.5) [gray, dashed];
    \draw (1.5,4)--(1.5,-2.5) [gray, dashed];
    \draw (2.5,4)--(2.5,-2.5) [gray, dashed];
    \draw (0,3.5) node {$\Ha_{0}$};
    \draw (1,3.5) node {$\II_{0}$};
    \draw (2,3.5) node {$\DD_{0}$};
    \draw (3,3.5) node {$\CC_{0}$};
    \node (v0) at (0,2) [blue] {$\bullet$};
    \node (v1) at (1,1) [red] {$\bullet$};
    \node (v2) at (0,1) [blue] {$\bullet$};
    \node (v3) at (1,-1) [red] {$\bullet$};
    \node (v4) at (0,0) [blue] {$\bullet$};
    \node (v5) at (0,-1) [blue] {$\bullet$};
    \node (v6) at (0,-2) [blue] {$\bullet$};
    \draw (v0) [above] node {$v_{0}$};
    \draw (v2) [above] node {$v_{2}$};
    \draw (v4) [below] node {$v_{4}$};
    \draw (v5) [below] node {$v_{5}$};
    \draw (v6) [below] node {$v_{6}$};
    \draw (v1) [above] node {$v_{1}$};
    \draw (v3) [below] node {$v_{3}$};
    \draw (v1)--(v0);
    \draw (v1)--(v2);
    \draw (v1)--(v5);
    \draw (v3)--(v0);
    \draw (v3)--(v4);
    \draw (v3)--(v5);
    \draw (v3)--(v6);
    \draw (v0) .. controls (-0.66,1.5) and (-0.66,0.5) .. (v4);
    \draw (v0) .. controls (-1,1.5) and (-1,-0.5) .. (v5);
    \draw (v2)--(v4);
    \draw (v2) .. controls (-0.66,0.5) and (-0.66,-0.5) .. (v5);
    \draw (v2) .. controls (-1,0.5) and (-1,-1.5) .. (v6);
\end{scope}
\begin{scope}[shift={(5,0)}]
	\draw (0.5,4)--(0.5,-2.5) [gray, dashed];
	\draw (1.5,4)--(1.5,-2.5) [gray, dashed];
	\draw (2.5,4)--(2.5,-2.5) [gray, dashed];
	\draw (0,3.5) node {$\Ha_{1}$};
	\draw (1,3.5) node {$\II_{1}$};
	\draw (2,3.5) node {$\DD_{1}$};
	\draw (3,3.5) node {$\CC_{1}$};
	\node (v0) at (1,2) [red] {$\bullet$};
	\node (v1) at (2,1) {$\bullet$};
	\node (v2) at (1,1) [red] {$\bullet$};
	\node (v3) at (2,-1) {$\bullet$};
	\node (v4) at (1,0) [red] {$\bullet$};
	\node (v5) at (1,-1) [red] {$\bullet$};
	\node (v6) at (1,-2) [red] {$\bullet$};
	\draw (v0) [above] node {$v_{0}$};
	\draw (v2) [above] node {$v_{2}$};
	\draw (v4) [below] node {$v_{4}$};
	\draw (v5) [below] node {$v_{5}$};
	\draw (v6) [below] node {$v_{6}$};
	\draw (v1) [above] node {$v_{1}$};
	\draw (v3) [below] node {$v_{3}$};
	\draw (v1)--(v0);
	\draw (v1)--(v2);
	\draw (v1)--(v5);
	\draw (v3)--(v0);
	\draw (v3)--(v4);
	\draw (v3)--(v5);
	\draw (v3)--(v6);
	\draw (v0) .. controls (0.33,1.5) and (0.33,0.5) .. (v4);
	\draw (v0) .. controls (0,1.5) and (0,-0.5) .. (v5);
	\draw (v2)--(v4);
	\draw (v2) .. controls (0.33,0.5) and (0.33,-0.5) .. (v5);
	\draw (v2) .. controls (0,0.5) and (0,-1.5) .. (v6);
\end{scope}
\begin{scope}[shift={(10,0)}]
	\draw (0.5,4)--(0.5,-2.5) [gray, dashed];
	\draw (1.5,4)--(1.5,-2.5) [gray, dashed];
	\draw (2.5,4)--(2.5,-2.5) [gray, dashed];
	\draw (0,3.5) node {$\Ha_{2}$};
	\draw (1,3.5) node {$\II_{2}$};
	\draw (2,3.5) node {$\DD_{2}$};
	\draw (3,3.5) node {$\CC_{2}$};
	\node (v0) at (2,2) [] {$\bullet$};
	\node (v1) at (1,1) [red] {$\bullet$};
	\node (v2) at (2,1) [] {$\bullet$};
	\node (v3) at (1,-1) [red] {$\bullet$};
	\node (v4) at (2,0) [] {$\bullet$};
	\node (v5) at (2,-1) [] {$\bullet$};
	\node (v6) at (2,-2) [] {$\bullet$};
	\draw (v0) [above] node {$v_{0}$};
	\draw (v2) [above] node {$v_{2}$};
	\draw (v4) [below] node {$v_{4}$};
	\draw (v5) [below] node {$v_{5}$};
	\draw (v6) [below] node {$v_{6}$};
	\draw (v1) [above] node {$v_{1}$};
	\draw (v3) [below] node {$v_{3}$};
	\draw (v1)--(v0);
	\draw (v1)--(v2);
	\draw (v1)--(v5);
	\draw (v3)--(v0);
	\draw (v3)--(v4);
	\draw (v3)--(v5);
	\draw (v3)--(v6);
	\draw (v0) .. controls (2.66,1.5) and (2.66,0.5) .. (v4);
	\draw (v0) .. controls (3,1.5) and (3,-0.5) .. (v5);
	\draw (v2)--(v4);
	\draw (v2) .. controls (2.66,0.5) and (2.66,-0.5) .. (v5);
	\draw (v2) .. controls (3,0.5) and (3,-1.5) .. (v6);
\end{scope}
\begin{scope}[shift={(15,0)}]
	\draw (0.5,4)--(0.5,-2.5) [gray, dashed];
	\draw (1.5,4)--(1.5,-2.5) [gray, dashed];
	\draw (2.5,4)--(2.5,-2.5) [gray, dashed];
	\draw (0,3.5) node {$\Ha_{3}$};
	\draw (1,3.5) node {$\II_{3}$};
	\draw (2,3.5) node {$\DD_{3}$};
	\draw (3,3.5) node {$\CC_{3}$};
	\node (v0) at (1,2) [red] {$\bullet$};
	\node (v1) at (2,0.5) [] {$\bullet$};
	\node (v2) at (3,1) [blue] {$\bullet$};
	\node (v3) at (2,-1) [] {$\bullet$};
	\node (v4) at (3,0) [blue] {$\bullet$};
	\node (v5) at (1,-1) [red] {$\bullet$};
	\node (v6) at (3,-2) [blue] {$\bullet$};
	\draw (v0) [above] node {$v_{0}$};
	\draw (v2) [above] node {$v_{2}$};
	\draw (v4) [below] node {$v_{4}$};
	\draw (v5) [below] node {$v_{5}$};
	\draw (v6) [below] node {$v_{6}$};
	\draw (v1) [left] node {$v_{1}$};
	\draw (v3) [below] node {$v_{3}$};
	\draw (v1)--(v0);
	\draw (v1)--(v2);
	\draw (v1)--(v5);
	\draw (v3)--(v0);
	\draw (v3)--(v4);
	\draw (v3)--(v5);
	\draw (v3)--(v6);
	\draw (v0)--(v4);
	\draw (v0) .. controls (0.33,1.5) and (0.33,-0.5) .. (v5);
	\draw (v2)--(v4);
	\draw (v2)--(v5);
	\draw (v2) .. controls (3.66,0.5) and (3.66,-1.5) .. (v6);
\end{scope}
\end{tikzpicture}

\vspace{1pc}

\begin{tikzpicture}[scale=0.74, transform shape,every node/.style={inner sep=1pt, font=\small}]
\draw (19.5,1.5) node {$\mathbf{R=\Delta=4}$};
\begin{scope}[shift={(0,0)}]
	\draw (0.5,4)--(0.5,-2.5) [gray, dashed];
	\draw (1.5,4)--(1.5,-2.5) [gray, dashed];
	\draw (2.5,4)--(2.5,-2.5) [gray, dashed];
	\draw (0,3.5) node {$\Ha_{0}$};
	\draw (1,3.5) node {$\II_{0}$};
	\draw (2,3.5) node {$\DD_{0}$};
	\draw (3,3.5) node {$\CC_{0}$};
	\node (v0) at (0,2) [blue] {$\bullet$};
	\node (v1) at (1,1) [red] {$\bullet$};
	\node (v2) at (0,1) [blue] {$\bullet$};
	\node (v3) at (1,-1) [red] {$\bullet$};
	\node (v4) at (0,0) [blue] {$\bullet$};
	\node (v5) at (0,-1) [blue] {$\bullet$};
	\node (v6) at (0,-2) [blue] {$\bullet$};
	\draw (v0) [above] node {$v_{0}$};
	\draw (v2) [above] node {$v_{2}$};
	\draw (v4) [below] node {$v_{4}$};
	\draw (v5) [below] node {$v_{5}$};
	\draw (v6) [below] node {$v_{6}$};
	\draw (v1) [above] node {$v_{1}$};
	\draw (v3) [below] node {$v_{3}$};
	\draw (v1)--(v0);
	\draw (v1)--(v2);
	\draw (v1)--(v5);
	\draw (v3)--(v0);
	\draw (v3)--(v4);
	\draw (v3)--(v5);
	\draw (v3)--(v6);
	\draw (v0) .. controls (-0.66,1.5) and (-0.66,0.5) .. (v4);
	\draw (v0) .. controls (-1,1.5) and (-1,-0.5) .. (v5);
	\draw (v2)--(v4);
	\draw (v2) .. controls (-0.66,0.5) and (-0.66,-0.5) .. (v5);
	\draw (v2) .. controls (-1,0.5) and (-1,-1.5) .. (v6);
\end{scope}
\begin{scope}[shift={(5,0)}]
	\draw (0.5,4)--(0.5,-2.5) [gray, dashed];
	\draw (1.5,4)--(1.5,-2.5) [gray, dashed];
	\draw (2.5,4)--(2.5,-2.5) [gray, dashed];
	\draw (0,3.5) node {$\Ha_{1}$};
	\draw (1,3.5) node {$\II_{1}$};
	\draw (2,3.5) node {$\DD_{1}$};
	\draw (3,3.5) node {$\CC_{1}$};
	\node (v0) at (1,2) [red] {$\bullet$};
	\node (v1) at (2,1) {$\bullet$};
	\node (v2) at (1,1) [red] {$\bullet$};
	\node (v3) at (2,-1) {$\bullet$};
	\node (v4) at (1,0) [red] {$\bullet$};
	\node (v5) at (1,-1) [red] {$\bullet$};
	\node (v6) at (1,-2) [red] {$\bullet$};
	\draw (v0) [above] node {$v_{0}$};
	\draw (v2) [above] node {$v_{2}$};
	\draw (v4) [below] node {$v_{4}$};
	\draw (v5) [below] node {$v_{5}$};
	\draw (v6) [below] node {$v_{6}$};
	\draw (v1) [above] node {$v_{1}$};
	\draw (v3) [below] node {$v_{3}$};
	\draw (v1)--(v0);
	\draw (v1)--(v2);
	\draw (v1)--(v5);
	\draw (v3)--(v0);
	\draw (v3)--(v4);
	\draw (v3)--(v5);
	\draw (v3)--(v6);
	\draw (v0) .. controls (0.33,1.5) and (0.33,0.5) .. (v4);
	\draw (v0) .. controls (0,1.5) and (0,-0.5) .. (v5);
	\draw (v2)--(v4);
	\draw (v2) .. controls (0.33,0.5) and (0.33,-0.5) .. (v5);
	\draw (v2) .. controls (0,0.5) and (0,-1.5) .. (v6);
\end{scope}
\begin{scope}[shift={(10,0)}]
	\draw (0.5,4)--(0.5,-2.5) [gray, dashed];
	\draw (1.5,4)--(1.5,-2.5) [gray, dashed];
	\draw (2.5,4)--(2.5,-2.5) [gray, dashed];
	\draw (0,3.5) node {$\Ha_{2}$};
	\draw (1,3.5) node {$\II_{2}$};
	\draw (2,3.5) node {$\DD_{2}$};
	\draw (3,3.5) node {$\CC_{2}$};
	\node (v0) at (2,2) [] {$\bullet$};
	\node (v1) at (3,0.5) [blue] {$\bullet$};
	\node (v2) at (2,1) [] {$\bullet$};
	\node (v3) at (1,0) [red] {$\bullet$};
	\node (v4) at (2,0) [] {$\bullet$};
	\node (v5) at (2,-1) [] {$\bullet$};
	\node (v6) at (2,-2) [] {$\bullet$};
	\draw (v0) [above] node {$v_{0}$};
	\draw (v2) [above] node {$v_{2}$};
	\draw (v4) [below] node {$v_{4}$};
	\draw (v5) [below] node {$v_{5}$};
	\draw (v6) [below] node {$v_{6}$};
	\draw (v1) [above] node {$v_{1}$};
	\draw (v3) [left] node {$v_{3}$};
	\draw (v1)--(v0);
	\draw (v1)--(v2);
	\draw (v1)--(v5);
	\draw (v3)--(v0);
	\draw (v3)--(v4);
	\draw (v3)--(v5);
	\draw (v3)--(v6);
	\draw (v0) .. controls (1.33,1.5) and (1.33,0.5) .. (v4);
	\draw (v0) .. controls (1,1.5) and (1,-0.5) .. (v5);
	\draw (v2)--(v4);
	\draw (v2) .. controls (1.33,0.5) and (1.33,-0.5) .. (v5);
	\draw (v2) .. controls (1,0.5) and (1,-1.5) .. (v6);
\end{scope}
\begin{scope}[shift={(15,0)}]
	\draw (0.5,4)--(0.5,-2.5) [gray, dashed];
	\draw (1.5,4)--(1.5,-2.5) [gray, dashed];
	\draw (2.5,4)--(2.5,-2.5) [gray, dashed];
	\draw (0,3.5) node {$\Ha_{3}$};
	\draw (1,3.5) node {$\II_{3}$};
	\draw (2,3.5) node {$\DD_{3}$};
	\draw (3,3.5) node {$\CC_{3}$};
	\node (v0) at (3,2) [blue] {$\bullet$};
	\node (v1) at (3,1) [blue] {$\bullet$};
	\node (v2) at (3,0) [blue] {$\bullet$};
	\node (v3) at (2,-0.5) [] {$\bullet$};
	\node (v4) at (3,-1) [blue] {$\bullet$};
	\node (v5) at (3,-2) [blue] {$\bullet$};
	\node (v6) at (3,-3) [blue] {$\bullet$};
	\draw (v0) [right] node {$v_{0}$};
	\draw (v2) [right] node {$v_{2}$};
	\draw (v4) [right] node {$v_{4}$};
	\draw (v5) [right] node {$v_{5}$};
	\draw (v6) [right] node {$v_{6}$};
	\draw (v1) [right] node {$v_{1}$};
	\draw (v3) [left] node {$v_{3}$};
	\draw (v1)--(v0);
	\draw (v1)--(v2);
	\draw (v1) .. controls (2.33,0.5) and (2.33,-1.5) ..(v5);
	\draw (v3)--(v0);
	\draw (v3)--(v4);
	\draw (v3)--(v5);
	\draw (v3)--(v6);
	\draw (v0) .. controls (2.33,1.5) and (2.33,-0.5) .. (v4);
	\draw (v0) .. controls (2,1.5) and (2,-1.5) .. (v5);
	\draw (v2)--(v4);
	\draw (v2) .. controls (2.33,-0.5) and (2.33,-1.5) .. (v5);
	\draw (v2) .. controls (2,-0.5) and (2,-2.5) .. (v6);
\end{scope}
\end{tikzpicture}
\caption{\label{fig:aristasmalas}States sequence of a graph for $R=\Delta=4$ and $R=2$.}
\end{figure}
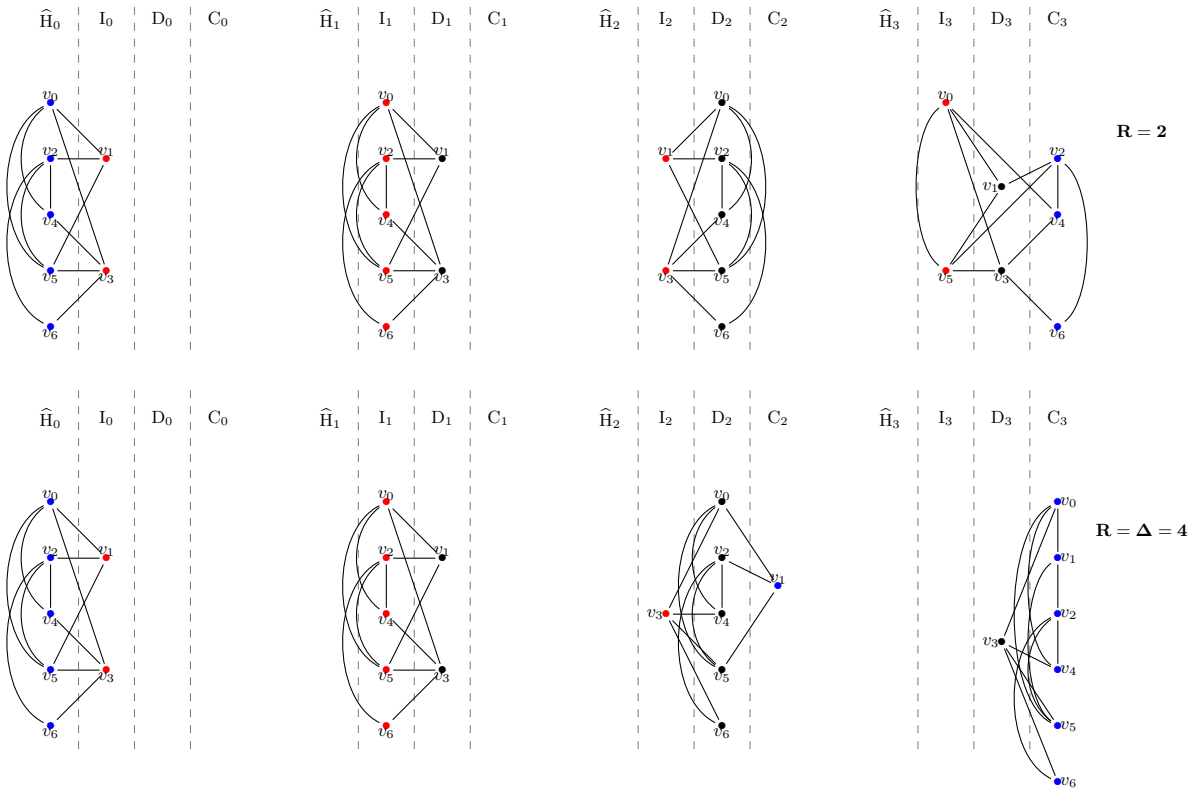
\end{center}
\vspace{-0.3in}

Also, by Proposition \ref{prop:reinfection} we know the time in which the infection becomes extinct if $\DD_{t}\cap \II_{t+1}=\emptyset$ for every $t$. In fact, let $f_{0}\in \Omega_{0}(G)\backslash\{\boldsymbol{0}_{G}\}$. If in a state sequence $(G,f_{0},R)$ there no exists vertices such that $v\in \DD_{t}\cap \II_{t+1}$, then each vertex $v$ is infected only at time $t_{v}=dist(v,\II_{0})$, where $dist(v,\II_{0})$ is the distance from $v$ to the set $\II_{0}$. Therefore, the infection becomes extinct at time $\max\{t_{v}:v\in V\}+2$.

\begin{theorem}\label{thm:HIVf0}
Let $G$ be a graph and let $f_0\in\Omega_0(G)$. Then $\HIV(G,f_0)\leq\Delta(G)+1.$
\end{theorem}

\begin{proof}
Let $R\geq\Delta(G)+1$. Since every vertex has at most $\Delta(G)<R$ infected neighbors, a dead vertex can never be replaced by an infected vertex. Hence every infected vertex becomes dead in one step and cured in the following step.

By Lemma \ref{Lemma:Estructura}, once a vertex becomes cured, it remains cured forever. Therefore each vertex can be infected at most once. Since $G$ has finitely many vertices, only finitely many infection events can occur. Consequently, after finitely many steps no infected vertices remain, and one step later the system reaches the all-healthy state. Thus every $R\geq\Delta(G)+1$ belongs to $\mathcal{E}(G,f_0)$, which proves the result.
\end{proof}

\begin{lemma}\label{lem:R1}
If $G$ contains at least one edge, then $1\in\mathcal{L}(G)$. Consequently, $\hiv(G)\geq 2$.
\end{lemma}

\begin{proof}
Let $uv\in E(G)$ and consider the admissible initial state $f_0$ defined by $f_0(u)=1$ and $f_0(x)=0$ for every $x\neq u$. For $R=1$, the vertices $u$ and $v$ alternate between the infected and dead states. More precisely, $f_{2k}(u)=1$, $f_{2k+1}(u)=2$ for every $k\geq 0$, while $f_{2k+1}(v)=1$, $f_{2k+2}(v)=2$. Thus the infection never becomes extinct. Hence $1\in\mathcal{L}(G,f_{0})$, so $1\in\mathcal{L}(G)$ and $\hiv(G)\geq 2$.
\end{proof}

\begin{theorem}\label{thm:bound}
If $G$ contains at least one edge, then $2\leq\hiv(G)\leq\HIV(G)\leq\Delta(G)+1$.
\end{theorem}

\begin{proof}
The lower bound follows from Lemma \ref{lem:R1}. The inequality $\hiv(G)\leq\HIV(G)$ follows immediately from the definitions. Finally, Theorem \ref{thm:HIVf0} holds for every admissible initial state, and therefore $\HIV(G)=\max_{f_0\in\Omega_0(G)} \HIV(G,f_0)\leq\Delta(G)+1$.
\end{proof}

For bipartite graphs, the lower bound can be improved.

\begin{theorem}\label{hivBip}
Let $G$ be a bipartite graph with minimum degree $\delta(G)\geq 1$. Then $\hiv(G)\geq\delta(G)+1$.
\end{theorem}

\begin{proof}
Let $A$ and $B$ be the partite sets of $G$, and let $1\leq R\leq\delta(G)$. Consider the admissible initial state in which every vertex of $A$ is infected and every vertex of $B$ is healthy.

At time $1$, every vertex of $A$ is dead and every vertex of $B$ is infected. Since each vertex of $A$ has at least $\delta(G)\geq R$ neighbors in $B$, every vertex of $A$ is reinfected at time $2$, while every vertex of $B$ becomes dead. Analogously, at time $3$ the vertices in $A$ becomes dead and the vertices in $B$ are reinfected. Thus the two partite sets alternate indefinitely between the infected and dead states. Hence $\{1,\ldots,\delta(G)\}\subseteq\mathcal{L}(G)$, and therefore $\hiv(G)\geq\delta(G)+1$.
\end{proof}

We now characterize the graphs that attain the upper bound in Theorem \ref{thm:bound}.

\begin{theorem}\label{thm:bipartiteregular}
Let $G$ be a connected graph with at least one edge. Then $\HIV(G)=\Delta(G)+1$ if and only if $G$ is bipartite and regular. In this case, $\hiv(G)=\HIV(G)=\Delta(G)+1$.
\end{theorem}

\begin{proof}
Suppose first that $G$ is bipartite and $\Delta(G)$-regular. From Theorem \ref{hivBip} and the fact that $\delta(G)+1=\Delta(G)+1$, we have that $\delta(G)+1\le\hiv(G)\le \HIV(G)\le \Delta(G)+1.$ Therefore the equality holds for bipartite regular graphs.

Conversely, suppose that $G$ is not both bipartite and regular. We will prove that $R=\Delta(G)$ guarantees the extinction of the infection for every admissible initial state. Together with Theorem \ref{thm:HIVf0}, this will imply $\HIV(G)\leq\Delta(G).$

Fix an admissible initial state $f_0$. If $f_0=\boldsymbol{0}_g$, the result is immediate. Suppose that at least one vertex is initially infected. We first show that every vertex is infected at least once, and a cured vertex eventually appears.  

Let $S=\{v\in V(G): v\in \II_t \text{ for some }t\geq 0\}.$ We first claim that $S=V(G)$. Otherwise, since $G$ is connected and $S\neq\emptyset$, there exists an edge $xy$ with $x\in S$ and $y\notin S$. Since $x\in S$, there exists a time $t$ such that $x\in\II_t$. On the other hand, $y\notin S$ implies that $y$ has never been infected, and hence $y\in\HH_t$. Therefore, $y\in\II_{t+1}$, contradicting $y\notin S$. Thus every vertex of $G$ is infected at some time.

Assume first that $G$ is not regular. Choose a vertex $v$ with $d(v)<\Delta(G)$. By the previous claim, there exists a time $t$ such that $v\in\II_t$. If a cured vertex appears before time $t$, there is nothing to prove. Otherwise, $v\in \DD_{t+1}$ and in the following step it becomes cured since $d(v)<\Delta(G)=R$.

Now assume that $G$ is not bipartite. Then $G$ contains an odd cycle $\mathcal{C}$. If no cured vertex appears, since every vertex must be infected at least once, when that happens, that vertex remains in $I_t\cup D_t$ for each subsequent time $t$. Hence there exists a time $t_0$ for which every vertex belongs to $\II_{t_0}\cup\DD_{t_0}$. Since $\mathcal{C}$ is odd, two consecutive vertices of $\mathcal{C}$ must have the same state at time $t_{0}$. If both are dead, then each one has at most $\Delta(G)-1<R$ infected neighbors, so each one becomes cured in the next step. If both are infected, they both become dead in the next step, and the previous argument applies one step later. Thus a cured vertex eventually appears.

In any case, let $t_0$ be a time such that $\CC_{t_0}\neq\emptyset$. By Observation \ref{obs:cured}, $\CC_t\subseteq\CC_{t+1}$. Suppose that $\CC_t\neq V(G)$. Since $G$ is connected, there exists an edge joining a vertex of $\CC_t$ to at least one vertex outside of $\CC_t$, say $v$. By Lemma \ref{Lemma:Estructura}, $v\in \DD_t$. Since it has a cured neighbor, it has at most $\Delta(G)-1<R$ infected neighbors, and hence it belongs to $\CC_{t+1}$. Consequently, $\CC_t\subsetneq\CC_{t+1}$ whenever $\CC_t\neq V(G)$. Since $G$ is finite, there exists $T$ such that $\CC_T=V(G)$. Thus the infection becomes extinct for every admissible initial state when $R=\Delta(G)$, and therefore $\HIV(G)\leq\Delta(G)$. This proves the converse.
\end{proof}

The parameters $\hiv(G)$ and $\HIV(G)$ may not coincide. In fact, their difference can be arbitrarily large.

\begin{theorem}\label{thm:unbounded-gap}
For every integer $k\geq 2$, there exists a connected graph $G$ such that $\HIV(G)-\hiv(G)=k.$
\end{theorem}

\begin{proof}
Let $F_k$ be the graph with vertices $u_1,\ldots,u_{2k-1},x,y,z$ and edges $E(F_k)=\{xy,yz\}\cup\{xu:u\in U\} \cup\{yu:u\in U\},$ where $U=\{u_1,\ldots,u_{2k-1}\}$ (see Figure \ref{fig:Fk}). Notice that $d(y)=2k+1$, $d(x)=2k$, $d(u_i)=2$ and $d(z)=1$, and therefore $\Delta(F_k)=2k+1.$
We will prove that $\hiv(F_k)=k+1$ and $\HIV(F_k)=2k+1$.

\begin{figure}[ht]
\centering
\begin{tikzpicture}[scale=0.8]
    \node (x) at (0,2) {$\bullet$};
    \node (y) at (0,0) {$\bullet$};
    \node (z) at (0,-2) {$\bullet$};
    \node (u1) at (1,1) {$\bullet$};
    \node (u2) at (2,1) {$\bullet$};
    \node (uu) at (4,1) {$\bullet$};
    \draw (3,1) node {$\dots$};
    \draw (x) [above left] node {$x$};
    \draw (y) [below left] node {$y$};
    \draw (z) [left] node {$z$};
    \draw (u1) [right] node {$u_1$};
    \draw (u2) [right] node {$u_2$};
    \draw (uu) [right] node {$u_{2k-1}$};
    \draw (z)--(y)--(x)--(u1)--(y)--(u2)--(x)--(uu)--(y);
\end{tikzpicture}
\caption{The graph $F_k$.}\label{fig:Fk}
\end{figure}
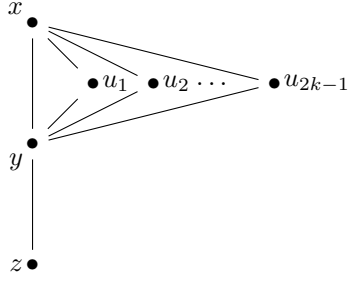

First, we show that $\{k+2,\ldots,2k\} \subseteq\mathcal{L}(F_k)$.
Fix an integer $R$ satisfying $k+2\leq R\leq 2k$. Let $U_0=\{u_1,\ldots,u_{R-2}\}$ and $U_1=\{u_{R-1},\ldots,u_{2k-1}\}$. Consider the initial state in which the vertices in $\{x,z\}\cup U_0$ are healthy and the vertices in $\{y\}\cup U_1$ are infected. The trajectory is given by
\begin{center}
$\begin{tblr}{colspec={*{19}{c}}, colsep=1pt, rowsep=0pt}
&\qquad& f_0&&f_1&&f_2&&f_3&&f_4&&f_5&&f_6&&f_7&
&f_8\\ \hline
x&\qquad&\HH_0&\to&\II_1&\to&\DD_2&\to&\HH_3&\to&\II_4&\to&\DD_5&\to&\HH_6&\to&\II_7&\to&\DD_8\\
y&\qquad&\II_0&\to&\DD_1&\to&\II_2&\to&\DD_3&\to&\HH_4&\to&\II_5&\to&\DD_6&\to&\HH_7&\to&\II_8\\
z&\qquad&\HH_0&\to&\II_1&\to&\DD_2&\to&\HH_3&\to&\HH_4&\to&\HH_5&\to&\II_6&\to&\DD_7&\to&\HH_8\\
U_0&\qquad&\HH_0&\to&\II_1&\to&\DD_2&\to&\HH_3&\to&\HH_4&\to&\II_5&\to&\DD_6&\to&\HH_7&\to&\II_8\\
U_1&\qquad&\II_0&\to&\DD_1&\to&\HH_2&\to&\II_3&\to&\DD_4&\to&\HH_5&\to&\II_6&\to&\DD_7&\to&\HH_8
\end{tblr}$
\end{center}
Since $f_5=f_8$, the trajectory is eventually periodic with period $3$. Thus every $R$ with $k+2\leq R\leq 2k$ belongs to $\mathcal{L}(F_k)$.

Now, we show that $\{1,\ldots,k\} \subseteq\mathcal{L}(F_k)$. The case $R=1$ follows from Lemma \ref{lem:R1}. For $2\leq R\leq k$, let $U_0=\{u_1,\ldots,u_{k-1}\}$ and $U_1=\{u_k,\ldots,u_{2k-1}\}$, and consider the initial state in which the vertices in $\{x,y,z\}\cup U_0$ are healthy and the vertices in $U_1$ are infected. The evolution is summarized in Table \ref{tab:flags1}.

\begin{table}[h!]
\begin{multicols}{2}\begin{center}
$\begin{tblr}{colspec={c *{13}{c}}, colsep=0pt, rowsep=0pt}
\SetCell[c=14]{c}\text{If }R=2\\\\
\qquad\qquad & f_0 & & f_1 & & f_2 & & f_3 & & f_4 & & f_5 & & f_6 \\
\hline
x  & \HH_0 & \to & \II_1 & \to & \DD_2 & \to & \II_3 & \to & \DD_4 & \to & \II_5 & \to & \DD_6 \\
y   & \HH_0 & \to & \II_1 & \to & \DD_2 & \to & \II_3 & \to & \DD_4 & \to & \II_5 & \to & \DD_6 \\
z   & \HH_0 & \to & \HH_1 & \to & \II_2 & \to & \DD_3 & \to & \HH_4 & \to & \HH_5 & \to & \II_6 \\
U_0 & \HH_0 & \to & \HH_1 & \to & \II_2 & \to & \DD_3 & \to & \II_4 & \to & \DD_5 & \to & \II_6 \\
U_1 & \II_0 & \to & \DD_1 & \to & \II_2 & \to & \DD_3 & \to & \II_4 & \to & \DD_5 & \to & \II_6
\end{tblr}$
\end{center}

\begin{center}
$\begin{tblr}{colspec={c *{11}{c}}, colsep=0pt, rowsep=0pt}
\SetCell[c=12]{c}\text{If }3\leq R\leq k-1\\\\
\qquad\qquad& f_0 & & f_1 & & f_2 & & f_3 & & f_4 & & f_5 \\
\hline
x   & \HH_0 & \to & \II_1 & \to & \DD_2 & \to & \II_3 & \to & \DD_4 & \to & \II_5 \\
y   & \HH_0 & \to & \II_1 & \to & \DD_2 & \to & \II_3 & \to & \DD_4 & \to & \II_5 \\
z & \HH_0 & \to & \HH_1 & \to & \II_2 & \to & \DD_3 & \to & \HH_4 & \to & \HH_5 \\
U_0 & \HH_0 & \to & \HH_1 & \to & \II_2 & \to & \DD_3 & \to & \HH_4 & \to & \HH_5 \\
U_1 & \II_0 & \to & \DD_1 & \to & \HH_2 & \to & \HH_3 & \to & \II_4 & \to & \DD_5
\end{tblr}$
\end{center}
\end{multicols}

\begin{center}
$\begin{tblr}{colspec={c *{17}{c}}, colsep=0pt, rowsep=0pt}
\SetCell[c=18]{c}\text{If }R=k\\\\
\qquad\qquad& f_0 & & f_1 & & f_2 & & f_3 & & f_4 & & f_5 & & f_6 & & f_7 & & f_8 \\
\hline
x   & \HH_0 & \to & \II_1 & \to & \DD_2 & \to & \HH_3 & \to & \II_4 & \to & \DD_5 & \to & \II_6 & \to & \DD_7 & \to & \II_8 \\
y   & \HH_0 & \to & \II_1 & \to & \DD_2 & \to & \II_3 & \to & \DD_4 & \to & \II_5 & \to & \DD_6 & \to & \HH_7 & \to & \II_8 \\
z   & \HH_0 & \to & \HH_1 & \to & \II_2 & \to & \DD_3 & \to & \HH_4 & \to & \HH_5 & \to & \II_6 & \to & \DD_7 & \to & \HH_8 \\
U_0 & \HH_0 & \to & \HH_1 & \to & \II_2 & \to & \DD_3 & \to & \HH_4 & \to & \II_5 & \to & \DD_6 & \to & \HH_7 & \to & \HH_8 \\
U_1 & \II_0 & \to & \DD_1 & \to & \HH_2 & \to & \HH_3 & \to & \II_4 & \to & \DD_5 & \to & \HH_6 & \to & \II_7 & \to & \DD_8
\end{tblr}$
\end{center}
\caption{\label{tab:flags1} Latency in $F_{k}$ for $2\le R\le k$.}
\end{table}

For $R=2$, we have $f_2=f_6$; for $3\leq R\leq k-1$, we have $f_1=f_5$; and for $R=k$, we have $f_1=f_8$. Hence the infection is latent for every $1\leq R\leq k$.

Since $F_k$ is connected and is neither bipartite nor regular, Theorem \ref{thm:bipartiteregular} gives $\HIV(F_k)\leq\Delta(F_k)=2k+1$. On the other hand, $2k\in\mathcal{L}(F_k)$. Therefore $\HIV(F_k)=2k+1$.

It remains to prove that $k+1\in\mathcal{E}(F_k)$. Fix $R=k+1$. For an admissible initial state $f_0$, let $U_0=U\cap \HH_{0}$ and $U_1=U\cap\II_{0}$. Since all vertices in $U$ have the same neighborhood, the trajectory depends only on $|U_0|$ and on the initial states of $x$, $y$, and $z$. Table \ref{tab:flags2} gives the extinction time $t_{\mathrm{ext}}$ for every admissible initial state.

\begin{table}
\begin{center}
$\begin{tblr}{colspec={c||c|c|c|c|c|c|c|c||c|c|c|c|c|c|c|c||c|c|c|c|c|c|c|c}, colsep=4pt, rowsep=0pt}
R=k+1&
\SetCell[c=8]{c}|U_0|=0&&&&&&&&
\SetCell[c=8]{c}1\leq |U_0|\leq k-2&&&&&&&&
\SetCell[c=8]{c}|U_0|=k-1\\ \hline
z&\HH&\HH&\HH&\HH&\II&\II&\II&\II&
\HH&\HH&\HH&\HH&\II&\II&\II&\II&
\HH&\HH&\HH&\HH&\II&\II&\II&\II\\
y&\HH&\HH&\II&\II&\HH&\HH&\II&\II&
\HH&\HH&\II&\II&\HH&\HH&\II&\II&
\HH&\HH&\II&\II&\HH&\HH&\II&\II\\
x&\HH&\II&\HH&\II&\HH&\II&\HH&\II&
\HH&\II&\HH&\II&\HH&\II&\HH&\II&
\HH&\II&\HH&\II&\HH&\II&\HH&\II\\ \hline
t_{\mathrm{ext}}&4&4&3&3&3&3&3&2&
4&4&3&3&4&3&3&3&
4&4&10&3&4&3&3&3
\end{tblr}$

\vspace{0.1in}

$\begin{tblr}{colspec={c||c|c|c|c|c|c|c|c||c|c|c|c|c|c|c|c||c|c|c|c|c|c|c|c}, colsep=4pt, rowsep=0pt}
R=k+1&
\SetCell[c=8]{c}|U_0|=k&&&&&&&&
\SetCell[c=8]{c}k+1\leq |U_0|\leq 2k-2&&&&&&&&
\SetCell[c=8]{c}|U_0|=2k-1\\ \hline
z&\HH&\HH&\HH&\HH&\II&\II&\II&\II&
\HH&\HH&\HH&\HH&\II&\II&\II&\II&
\HH&\HH&\HH&\HH&\II&\II&\II&\II\\
y&\HH&\HH&\II&\II&\HH&\HH&\II&\II&
\HH&\HH&\II&\II&\HH&\HH&\II&\II&
\HH&\HH&\II&\II&\HH&\HH&\II&\II\\
x&\HH&\II&\HH&\II&\HH&\II&\HH&\II&
\HH&\II&\HH&\II&\HH&\II&\HH&\II&
\HH&\II&\HH&\II&\HH&\II&\HH&\II\\ \hline
t_{\mathrm{ext}}&15&8&15&14&4&8&15&3&
6&8&8&5&6&8&8&5&
0&4&4&4&6&4&5&5
\end{tblr}$
\end{center}
\caption{\label{tab:flags2} Extinction of the infection in $F_{k}$ for $R=k+1$. $t_{ext}$ denotes the first time at which each initial admissible state reaches the state $\boldsymbol{0}_{G}$.}
\end{table}

Thus $k+1\in\mathcal{E}(F_k)$. Since every $R\leq k$ belongs to $\mathcal{L}(F_k)$, it follows that $\hiv(F_k)=k+1.$ Therefore $\HIV(F_k)-\hiv(F_k)=(2k+1)-(k+1)=k$.\qedhere
\end{proof}

\section{\texorpdfstring{The $R$-state transition digraph}{The R-state transition digraph}}\label{sec:rState}

Given a graph $G$ and a replacement parameter $R$, in this section we study on the global dynamics induced by $R$ on the set of states $\{0,1,2\}^{V(G)}$. In Section \ref{sec:extinctionT} we denoted by $(f_t^{R,f_0})_{t\geq 0}$ the state sequence generated by Mukwembi's dynamics starting from admissible states $f_0$. In general we can get rid of the last condition and consider state sequences starting from any state in $\{0,1,2\}^{V(G)}$. Now, we define $T_R:\{0,1,2\}^{V(G)}\to \{0,1,2\}^{V(G)}$ the map that assigns to each state $f\in\{0,1,2\}^{V(G)}$, the state obtained by applying Mukwembi's dynamics described in Equation \ref{eq:rules}. We call $T_R$ the \emph{global transition map} of $G$ and $R$. Notice that, with this notation, we have $(f_t^{R,f_0})_{t\geq 0} = (f_0, T_R(f_0), T_R^2(f_0), \cdots)$. Using this notation, we define the main study object of this section.

The \emph{$R$-state transition digraph} of $G$, denoted by $\mathcal D_R(G)$, is the directed graph whose vertex set is $\{0,1,2\}^{V(G)}$, with an arc from $f$ to $g$ if and only if $g=T_R(f)$. Thus, $\mathcal D_R(G)$ is precisely the phase-space digraph of the Mukwembi cellular automaton: its vertices are the configurations of the automaton and its arcs represent one application of the global transition function $T_R$.

Since the arcs are defined by a map, any vertex of $\mathcal D_R(G)$ has out-degree exactly one. Moreover, the only vertex with a loop is $\boldsymbol{0}_G$. The structure of the weakly and strongly connected components is described in the following remark.

\begin{remark}\label{rem:digraphStructure}
    Let $D$ be a directed graph with all vertices having out-degree $1$. Every nontrivial strongly connected component of $D$ is a directed cycle. Moreover, every weakly connected component of a finite digraph in which every vertex has out-degree one contains a unique directed cycle, with directed in-trees attached to its vertices.
\end{remark}

Let $D_H$ be the weakly connected component of $D_R(G)$ containing $\boldsymbol{0}_G$. As a direct consequence of Remark \ref{rem:digraphStructure} we have the following observation.

\begin{observation}\label{obs:weakCom}
    Let $f_0\in\Omega_0(G)$ and let $R$ be a replacement parameter. The infection becomes extinct on $(G,f_0,R)$ if and only if $f_0\in D_H$. Consequently, $R$ guarantees the extinction of the infection on $G$ if and only if $\Omega_0(G)\subseteq D_H$. 
\end{observation}

Notice that there could be weakly connected components of $D_R(G)$ different from $D_H$ even when $R$ guarantees the extinction of the infection (see Figure \ref{fig:stateK3}). By Observation \ref{obs:weakCom} the infection is latent for $R$ if and only if there exists $f_0\in\Omega_0(G)$ with $f_0\not\in D_H$. In this case the weakly connected component of $f_0$ has its own unique directed cycle. In the following result, we show that this cycle can be of any length.

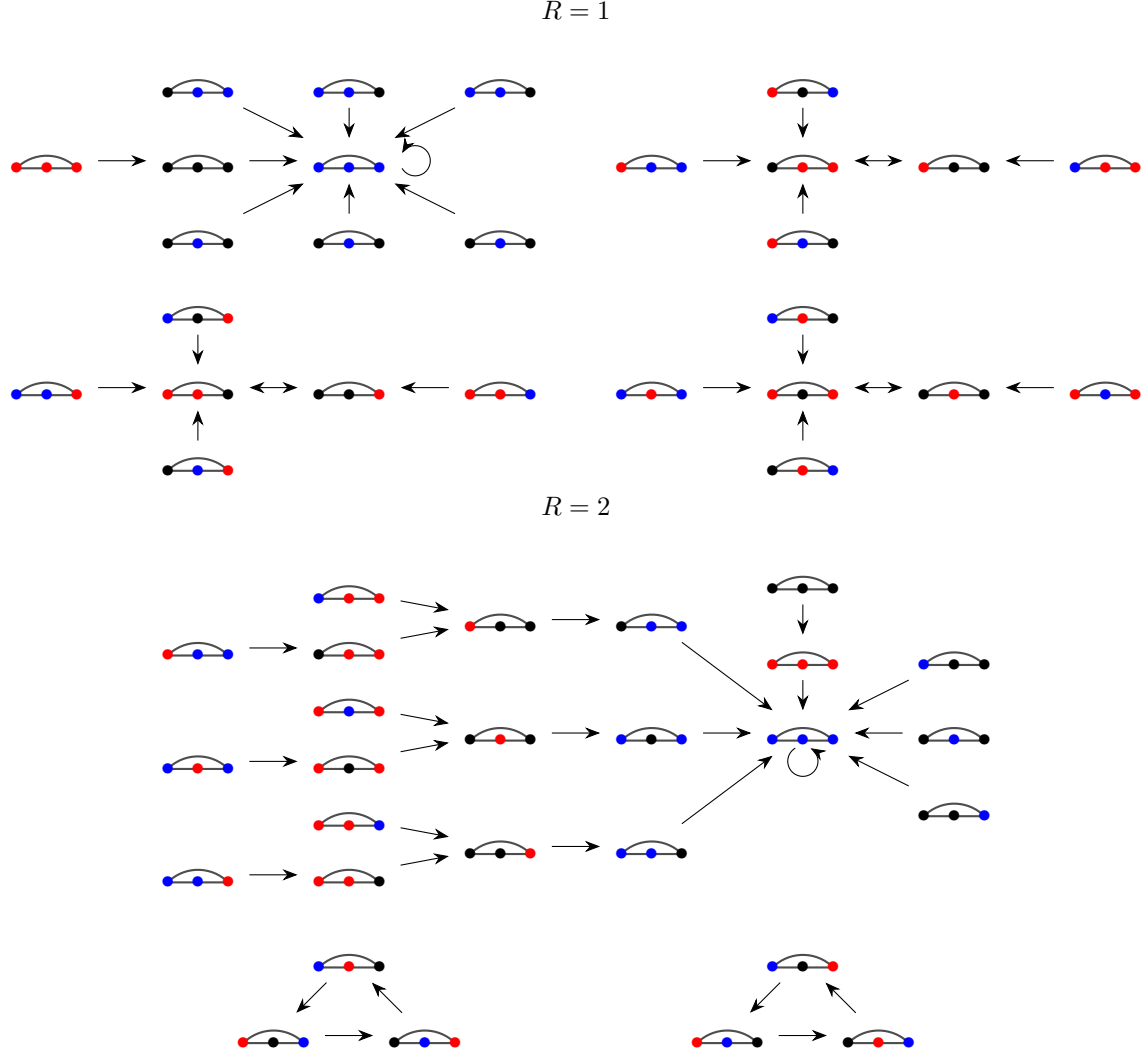
\begin{figure}[ht]
\centering
\begin{tikzpicture}[>={Stealth[scale=1.2]}]
    \draw (0,4) node {$R=1$};
    \begin{scope}[shift={(-3,2)}]
    \node (000a) at (0,0) [white] {$(0,0,0)$}; 
    \draw [black!70, thick] (-0.4,-0.1)--(0.4,-0.1);
    \draw [black!70, thick] (-0.4,-0.1) arc (135:45:0.567);
    \draw [blue] (0,-0.1) node {$\bullet$};
    \draw [blue] (0.4,-0.1) node {$\bullet$};
    \draw [blue] (-0.4,-0.1) node {$\bullet$};

    \node (002a) at (2,1) [white] {$(0,0,2)$};
    \draw [black!70, thick] (1.6,0.9)--(2.4,0.9);
    \draw [black!70, thick] (1.6,0.9) arc (135:45:0.567);
    \draw [blue]  (2.0,0.9) node {$\bullet$};
    \draw [black] (2.4,0.9) node {$\bullet$};
    \draw [blue]  (1.6,0.9) node {$\bullet$}; 
    
    \node (020a) at (0,1) [white] {$(0,2,0)$};
    \draw [black!70, thick] (-0.4,0.9)--(0.4,0.9);
    \draw [black!70, thick] (-0.4,0.9) arc (135:45:0.567);
    \draw [blue] (0,0.9) node {$\bullet$};
    \draw [black] (0.4,0.9) node {$\bullet$};
    \draw [blue] (-0.4,0.9) node {$\bullet$};

    \node (022a) at (2,-1) [white] {$(0,2,2)$};
    \draw [black!70, thick] (1.6,-1.1)--(2.4,-1.1);
    \draw [black!70, thick] (1.6,-1.1) arc (135:45:0.567);
    \draw [blue] (2,-1.1) node {$\bullet$};
    \draw [black] (2.4,-1.1) node {$\bullet$};
    \draw [black] (1.6,-1.1) node {$\bullet$};

    \node (200a) at (-2,1) [white] {$(2,0,0)$};
    \draw [black!70, thick] (-2.4,0.9)--(-1.6,0.9);
    \draw [black!70, thick] (-2.4,0.9) arc (135:45:0.567);
    \draw [blue] (-2,0.9) node {$\bullet$};
    \draw [blue] (-1.6,0.9) node {$\bullet$};
    \draw [black] (-2.4,0.9) node {$\bullet$};

    \node (202a) at (0,-1) [white] {$(2,0,2)$};
    \draw [black!70, thick] (-0.4,-1.1)--(0.4,-1.1);
    \draw [black!70, thick] (-0.4,-1.1) arc (135:45:0.567);
    \draw [blue] (0,-1.1) node {$\bullet$};
    \draw [black] (0.4,-1.1) node {$\bullet$};
    \draw [black] (-0.4,-1.1) node {$\bullet$};

    \node (220a) at (-2,-1) [white] {$(2,2,0)$};
    \draw [black!70, thick] (-2.4,-1.1)--(-1.6,-1.1);
    \draw [black!70, thick] (-2.4,-1.1) arc (135:45:0.567);
    \draw [blue] (-2,-1.1) node {$\bullet$};
    \draw [black] (-1.6,-1.1) node {$\bullet$};
    \draw [black] (-2.4,-1.1) node {$\bullet$};

    \node (222a) at (-2,0) [white] {$(2,2,2)$};
    \draw [black!70, thick] (-2.4,-0.1)--(-1.6,-0.1);
    \draw [black!70, thick] (-2.4,-0.1) arc (135:45:0.567);
    \draw [black] (-2,-0.1) node {$\bullet$};
    \draw [black] (-1.6,-0.1) node {$\bullet$};
    \draw [black] (-2.4,-0.1) node {$\bullet$};

    \node (111a) at (-4,0) [white] {$(1,1,1)$};
    \draw [black!70, thick] (-4.4,-0.1)--(-3.6,-0.1);
    \draw [black!70, thick] (-4.4,-0.1) arc (135:45:0.567);
    \draw [red] (-4,-0.1) node {$\bullet$};
    \draw [red] (-3.6,-0.1) node {$\bullet$};
    \draw [red] (-4.4,-0.1) node {$\bullet$};    \draw [->] (002a)--(000a);
    \draw [->] (020a)--(000a);
    \draw [->] (200a)--(000a);
    \draw [->] (202a)--(000a);
    \draw [->] (022a)--(000a);
    \draw [->] (220a)--(000a);
    \draw [->] (222a)--(000a);
    \draw [->] (111a)--(222a);
    \draw [->] (0.7,-0.1) arc (-150:150:0.2);
    \end{scope}
    \begin{scope}[shift={(5,2)}]
        \node (011a) at (2,0) [white] {$(0,1,1)$};
        \draw [black!70, thick] (1.6,-0.1)--(2.4,-0.1);
        \draw [black!70, thick] (1.6,-0.1) arc (135:45:0.567);
        \draw [red] (2,-0.1) node {$\bullet$};
        \draw [red] (2.4,-0.1) node {$\bullet$};
        \draw [blue] (1.6,-0.1) node {$\bullet$};

        \node (100a) at (-4,0) [white] {$(1,0,0)$};
        \draw [black!70, thick] (-4.4,-0.1)--(-3.6,-0.1);
        \draw [black!70, thick] (-4.4,-0.1) arc (135:45:0.567);
        \draw [blue] (-4,-0.1) node {$\bullet$};
        \draw [blue] (-3.6,-0.1) node {$\bullet$};
        \draw [red] (-4.4,-0.1) node {$\bullet$};

        \node (102a) at (-2,-1) [white] {$(1,0,2)$};
        \draw [black!70, thick] (-2.4,-1.1)--(-1.6,-1.1);
        \draw [black!70, thick] (-2.4,-1.1) arc (135:45:0.567);
        \draw [blue] (-2,-1.1) node {$\bullet$};
        \draw [black] (-1.6,-1.1) node {$\bullet$};
        \draw [red] (-2.4,-1.1) node {$\bullet$};

        \node (120a) at (-2,1) [white] {$(1,2,0)$};
        \draw [black!70, thick] (-2.4,0.9)--(-1.6,0.9);
        \draw [black!70, thick] (-2.4,0.9) arc (135:45:0.567);
        \draw [black] (-2,0.9) node {$\bullet$};
        \draw [blue] (-1.6,0.9) node {$\bullet$};
        \draw [red] (-2.4,0.9) node {$\bullet$};

        \node (122a) at (0,0) [white] {$(1,2,2)$};
        \draw [black!70, thick] (-0.4,-0.1)--(0.4,-0.1);
        \draw [black!70, thick] (-0.4,-0.1) arc (135:45:0.567);
        \draw [black] (0,-0.1) node {$\bullet$};
        \draw [black] (0.4,-0.1) node {$\bullet$};
        \draw [red] (-0.4,-0.1) node {$\bullet$};

        \node (211a) at (-2,0) [white] {$(2,1,1)$};
        \draw [black!70, thick] (-2.4,-0.1)--(-1.6,-0.1);
        \draw [black!70, thick] (-2.4,-0.1) arc (135:45:0.567);
        \draw [red] (-2,-0.1) node {$\bullet$};
        \draw [red] (-1.6,-0.1) node {$\bullet$};
        \draw [black] (-2.4,-0.1) node {$\bullet$};

        \draw [->] (100a)--(211a);
        \draw [->] (120a)--(211a);
        \draw [->] (102a)--(211a);
        \draw [->] (011a)--(122a);
        \draw [<->] (211a)--(122a);
    \end{scope}
    \begin{scope}[shift={(-3,-1)}]
        \node (110a) at (2,0) [white] {$(1,1,0)$};
        \draw [black!70, thick] (1.6,-0.1)--(2.4,-0.1);
        \draw [black!70, thick] (1.6,-0.1) arc (135:45:0.567);
        \draw [red] (2,-0.1) node {$\bullet$};
        \draw [blue] (2.4,-0.1) node {$\bullet$};
        \draw [red] (1.6,-0.1) node {$\bullet$};

        \node (001a) at (-4,0) [white] {$(0,0,1)$};
        \draw [black!70, thick] (-4.4,-0.1)--(-3.6,-0.1);
        \draw [black!70, thick] (-4.4,-0.1) arc (135:45:0.567);
        \draw [blue] (-4,-0.1) node {$\bullet$};
        \draw [red] (-3.6,-0.1) node {$\bullet$};
        \draw [blue] (-4.4,-0.1) node {$\bullet$};

        \node (201a) at (-2,-1) [white] {$(2,0,1)$};
        \draw [black!70, thick] (-2.4,-1.1)--(-1.6,-1.1);
        \draw [black!70, thick] (-2.4,-1.1) arc (135:45:0.567);
        \draw [blue] (-2,-1.1) node {$\bullet$};
        \draw [red] (-1.6,-1.1) node {$\bullet$};
        \draw [black] (-2.4,-1.1) node {$\bullet$};

        \node (021a) at (-2,1) [white] {$(0,2,1)$};
        \draw [black!70, thick] (-2.4,0.9)--(-1.6,0.9);
        \draw [black!70, thick] (-2.4,0.9) arc (135:45:0.567);
        \draw [black] (-2,0.9) node {$\bullet$};
        \draw [red] (-1.6,0.9) node {$\bullet$};
        \draw [blue] (-2.4,0.9) node {$\bullet$};

        \node (221a) at (0,0) [white] {$(2,2,1)$};
        \draw [black!70, thick] (-0.4,-0.1)--(0.4,-0.1);
        \draw [black!70, thick] (-0.4,-0.1) arc (135:45:0.567);
        \draw [black] (0,-0.1) node {$\bullet$};
        \draw [red] (0.4,-0.1) node {$\bullet$};
        \draw [black] (-0.4,-0.1) node {$\bullet$};

        \node (112a) at (-2,0) [white] {$(1,1,2)$};
        \draw [black!70, thick] (-2.4,-0.1)--(-1.6,-0.1);
        \draw [black!70, thick] (-2.4,-0.1) arc (135:45:0.567);
        \draw [red] (-2,-0.1) node {$\bullet$};
        \draw [black] (-1.6,-0.1) node {$\bullet$};
        \draw [red] (-2.4,-0.1) node {$\bullet$};

        \draw [->] (001a)--(112a);
        \draw [->] (201a)--(112a);
        \draw [->] (021a)--(112a);
        \draw [->] (110a)--(221a);
        \draw [<->] (112a)--(221a);
    \end{scope}
    \begin{scope}[shift={(5,-1)}]
        \node (101a) at (2,0) [white] {$(1,0,1)$};
        \draw [black!70, thick] (1.6,-0.1)--(2.4,-0.1);
        \draw [black!70, thick] (1.6,-0.1) arc (135:45:0.567);
        \draw [blue] (2,-0.1) node {$\bullet$};
        \draw [red] (2.4,-0.1) node {$\bullet$};
        \draw [red] (1.6,-0.1) node {$\bullet$};

        \node (010a) at (-4,0) [white] {$(0,1,0)$};
        \draw [black!70, thick] (-4.4,-0.1)--(-3.6,-0.1);
        \draw [black!70, thick] (-4.4,-0.1) arc (135:45:0.567);
        \draw [red] (-4,-0.1) node {$\bullet$};
        \draw [blue] (-3.6,-0.1) node {$\bullet$};
        \draw [blue] (-4.4,-0.1) node {$\bullet$};

        \node (210a) at (-2,-1) [white] {$(2,1,0)$};
        \draw [black!70, thick] (-2.4,-1.1)--(-1.6,-1.1);
        \draw [black!70, thick] (-2.4,-1.1) arc (135:45:0.567);
        \draw [red] (-2,-1.1) node {$\bullet$};
        \draw [blue] (-1.6,-1.1) node {$\bullet$};
        \draw [black] (-2.4,-1.1) node {$\bullet$};

        \node (012a) at (-2,1) [white] {$(0,1,2)$};
        \draw [black!70, thick] (-2.4,0.9)--(-1.6,0.9);
        \draw [black!70, thick] (-2.4,0.9) arc (135:45:0.567);
        \draw [red] (-2,0.9) node {$\bullet$};
        \draw [black] (-1.6,0.9) node {$\bullet$};
        \draw [blue] (-2.4,0.9) node {$\bullet$};

        \node (212a) at (0,0) [white] {$(2,1,2)$};
        \draw [black!70, thick] (-0.4,-0.1)--(0.4,-0.1);
        \draw [black!70, thick] (-0.4,-0.1) arc (135:45:0.567);
        \draw [red] (0,-0.1) node {$\bullet$};
        \draw [black] (0.4,-0.1) node {$\bullet$};
        \draw [black] (-0.4,-0.1) node {$\bullet$};

        \node (121a) at (-2,0) [white] {$(1,2,1)$};
        \draw [black!70, thick] (-2.4,-0.1)--(-1.6,-0.1);
        \draw [black!70, thick] (-2.4,-0.1) arc (135:45:0.567);
        \draw [black] (-2,-0.1) node {$\bullet$};
        \draw [red] (-1.6,-0.1) node {$\bullet$};
        \draw [red] (-2.4,-0.1) node {$\bullet$};

        \draw [->] (010a)--(121a);
        \draw [->] (210a)--(121a);
        \draw [->] (012a)--(121a);
        \draw [->] (101a)--(212a);
        \draw [<->] (121a)--(212a);
    \end{scope}  
\end{tikzpicture}
\begin{tikzpicture}[>={Stealth[scale=1.2]}]
    \draw (0,-3) node {$R=2$};
    \begin{scope}[shift={(1,-6)}]
        \node (001b) at (-6,-1.875) [white] {$(0,0,1)$};
        \draw [black!70, thick] (-6.4,-1.975)--(-5.6,-1.975);
        \draw [black!70, thick] (-6.4,-1.975) arc (135:45:0.567);
        \draw [blue] (-6,-1.975) node {$\bullet$};
        \draw [red] (-5.6,-1.975) node {$\bullet$};
        \draw [blue] (-6.4,-1.975) node {$\bullet$};

        \node (010b) at (-6,-0.375) [white] {$(0,1,0)$};
        \draw [black!70, thick] (-6.4,-0.475)--(-5.6,-0.475);
        \draw [black!70, thick] (-6.4,-0.475) arc (135:45:0.567);
        \draw [red] (-6,-0.475) node {$\bullet$};
        \draw [blue] (-5.6,-0.475) node {$\bullet$};
        \draw [blue] (-6.4,-0.475) node {$\bullet$};

        \node (100b) at (-6,1.125) [white] {$(1,0,0)$};
        \draw [black!70, thick] (-6.4,1.025)--(-5.6,1.025);
        \draw [black!70, thick] (-6.4,1.025) arc (135:45:0.567);
        \draw [blue] (-6,1.025) node {$\bullet$};
        \draw [blue] (-5.6,1.025) node {$\bullet$};
        \draw [red] (-6.4,1.025) node {$\bullet$};

        \node (011b) at (-4,1.875) [white] {$(0,1,1)$};
        \draw [black!70, thick] (-4.4,1.775)--(-3.6,1.775);
        \draw [black!70, thick] (-4.4,1.775) arc (135:45:0.567);
        \draw [red] (-4,1.775) node {$\bullet$};
        \draw [red] (-3.6,1.775) node {$\bullet$};
        \draw [blue] (-4.4,1.775) node {$\bullet$};

        \node (101b) at (-4,0.375) [white] {$(1,0,1)$};
        \draw [black!70, thick] (-4.4,0.275)--(-3.6,0.275);
        \draw [black!70, thick] (-4.4,0.275) arc (135:45:0.567);
        \draw [blue] (-4,0.275) node {$\bullet$};
        \draw [red] (-3.6,0.275) node {$\bullet$};
        \draw [red] (-4.4,0.275) node {$\bullet$};

        \node (110b) at (-4,-1.125) [white] {$(1,1,0)$};
        \draw [black!70, thick] (-4.4,-1.225)--(-3.6,-1.225);
        \draw [black!70, thick] (-4.4,-1.225) arc (135:45:0.567);
        \draw [red] (-4,-1.225) node {$\bullet$};
        \draw [blue] (-3.6,-1.225) node {$\bullet$};
        \draw [red] (-4.4,-1.225) node {$\bullet$};

        \node (112b) at (-4,-1.875) [white] {$(1,1,2)$};
        \draw [black!70, thick] (-4.4,-1.975)--(-3.6,-1.975);
        \draw [black!70, thick] (-4.4,-1.975) arc (135:45:0.567);
        \draw [red] (-4,-1.975) node {$\bullet$};
        \draw [black] (-3.6,-1.975) node {$\bullet$};
        \draw [red] (-4.4,-1.975) node {$\bullet$};

        \node (121b) at (-4,-0.375) [white] {$(1,2,1)$};
        \draw [black!70, thick] (-4.4,-0.475)--(-3.6,-0.475);
        \draw [black!70, thick] (-4.4,-0.475) arc (135:45:0.567);
        \draw [black] (-4,-0.475) node {$\bullet$};
        \draw [red] (-3.6,-0.475) node {$\bullet$};
        \draw [red] (-4.4,-0.475) node {$\bullet$};

        \node (211b) at (-4,1.125) [white] {$(2,1,1)$};
        \draw [black!70, thick] (-4.4,1.025)--(-3.6,1.025);
        \draw [black!70, thick] (-4.4,1.025) arc (135:45:0.567);
        \draw [red] (-4,1.025) node {$\bullet$};
        \draw [red] (-3.6,1.025) node {$\bullet$};
        \draw [black] (-4.4,1.025) node {$\bullet$};

        \node (122b) at (-2,1.5) [white] {$(1,2,2)$};
        \draw [black!70, thick] (-2.4,1.4)--(-1.6,1.4);
        \draw [black!70, thick] (-2.4,1.4) arc (135:45:0.567);
        \draw [black] (-2,1.4) node {$\bullet$};
        \draw [black] (-1.6,1.4) node {$\bullet$};
        \draw [red] (-2.4,1.4) node {$\bullet$};

        \node (212b) at (-2,0) [white] {$(2,1,2)$};
        \draw [black!70, thick] (-2.4,-0.1)--(-1.6,-0.1);
        \draw [black!70, thick] (-2.4,-0.1) arc (135:45:0.567);
        \draw [red] (-2,-0.1) node {$\bullet$};
        \draw [black] (-1.6,-0.1) node {$\bullet$};
        \draw [black] (-2.4,-0.1) node {$\bullet$};

        \node (221b) at (-2,-1.5) [white] {$(2,2,1)$};
        \draw [black!70, thick] (-2.4,-1.6)--(-1.6,-1.6);
        \draw [black!70, thick] (-2.4,-1.6) arc (135:45:0.567);
        \draw [black] (-2,-1.6) node {$\bullet$};
        \draw [red] (-1.6,-1.6) node {$\bullet$};
        \draw [black] (-2.4,-1.6) node {$\bullet$};

        \node (002b) at (0,-1.5) [white] {$(0,0,2)$};
        \draw [black!70, thick] (-0.4,-1.6)--(0.4,-1.6);
        \draw [black!70, thick] (-0.4,-1.6) arc (135:45:0.567);
        \draw [blue] (0,-1.6) node {$\bullet$};
        \draw [black] (0.4,-1.6) node {$\bullet$};
        \draw [blue] (-0.4,-1.6) node {$\bullet$};

        \node (020b) at (0,0) [white] {$(0,2,0)$};
        \draw [black!70, thick] (-0.4,-0.1)--(0.4,-0.1);
        \draw [black!70, thick] (-0.4,-0.1) arc (135:45:0.567);
        \draw [black] (0,-0.1) node {$\bullet$};
        \draw [blue] (0.4,-0.1) node {$\bullet$};
        \draw [blue] (-0.4,-0.1) node {$\bullet$};

        \node (200b) at (0,1.5) [white] {$(2,0,0)$};
        \draw [black!70, thick] (-0.4,1.4)--(0.4,1.4);
        \draw [black!70, thick] (-0.4,1.4) arc (135:45:0.567);
        \draw [blue] (0,1.4) node {$\bullet$};
        \draw [blue] (0.4,1.4) node {$\bullet$};
        \draw [black] (-0.4,1.4) node {$\bullet$};

        \node (000b) at (2,0) [white] {$(0,0,0)$};
        \draw [black!70, thick] (1.6,-0.1)--(2.4,-0.1);
        \draw [black!70, thick] (1.6,-0.1) arc (135:45:0.567);
        \draw [blue] (2,-0.1) node {$\bullet$};
        \draw [blue] (2.4,-0.1) node {$\bullet$};
        \draw [blue] (1.6,-0.1) node {$\bullet$};

        \node (022b) at (4,1) [white] {$(0,2,2)$};
        \draw [black!70, thick] (3.6,0.9)--(4.4,0.9);
        \draw [black!70, thick] (3.6,0.9) arc (135:45:0.567);
        \draw [black] (4,0.9) node {$\bullet$};
        \draw [black] (4.4,0.9) node {$\bullet$};
        \draw [blue] (3.6,0.9) node {$\bullet$};

        \node (111b) at (2,1) [white] {$(1,1,1)$};
        \draw [black!70, thick] (1.6,0.9)--(2.4,0.9);
        \draw [black!70, thick] (1.6,0.9) arc (135:45:0.567);
        \draw [red] (2,0.9) node {$\bullet$};
        \draw [red] (2.4,0.9) node {$\bullet$};
        \draw [red] (1.6,0.9) node {$\bullet$};

        \node (202b) at (4,0) [white] {$(2,0,2)$};
        \draw [black!70, thick] (3.6,-0.1)--(4.4,-0.1);
        \draw [black!70, thick] (3.6,-0.1) arc (135:45:0.567);
        \draw [blue] (4,-0.1) node {$\bullet$};
        \draw [black] (4.4,-0.1) node {$\bullet$};
        \draw [black] (3.6,-0.1) node {$\bullet$};

        \node (220b) at (4,-1) [white] {$(2,2,0)$};
        \draw [black!70, thick] (3.6,-1.1)--(4.4,-1.1);
        \draw [black!70, thick] (3.6,-1.1) arc (135:45:0.567);
        \draw [black] (4,-1.1) node {$\bullet$};
        \draw [blue] (4.4,-1.1) node {$\bullet$};
        \draw [black] (3.6,-1.1) node {$\bullet$};

        \node (222b) at (2,2) [white] {$(2,2,2)$};
        \draw [black!70, thick] (1.6,1.9)--(2.4,1.9);
        \draw [black!70, thick] (1.6,1.9) arc (135:45:0.567);
        \draw [black] (2,1.9) node {$\bullet$};
        \draw [black] (2.4,1.9) node {$\bullet$};
        \draw [black] (1.6,1.9) node {$\bullet$};

        \draw [->] (100b)--(211b);
        \draw [->] (211b)--(122b);
        \draw [->] (011b)--(122b);
        \draw [->] (122b)--(200b);
        \draw [->] (200b)--(000b);
        \draw [->] (010b)--(121b);
        \draw [->] (121b)--(212b);
        \draw [->] (101b)--(212b);
        \draw [->] (212b)--(020b);
        \draw [->] (020b)--(000b);
        \draw [->] (001b)--(112b);
        \draw [->] (112b)--(221b);
        \draw [->] (110b)--(221b);
        \draw [->] (221b)--(002b);
        \draw [->] (002b)--(000b);
        \draw [->] (222b)--(111b);
        \draw [->] (111b)--(000b);
        \draw [->] (022b)--(000b);
        \draw [->] (202b)--(000b);
        \draw [->] (220b)--(000b);
        \draw [->] (1.9,-0.2) arc (120:420:0.2);
    \end{scope}
    \begin{scope}[shift={(-3,-9)}]
        \node (012b) at (0,0) [white] {$(0,1,2)$};
        \draw [black!70, thick] (-0.4,-0.1)--(0.4,-0.1);
        \draw [black!70, thick] (-0.4,-0.1) arc (135:45:0.567);
        \draw [red] (0,-0.1) node {$\bullet$};
        \draw [black] (0.4,-0.1) node {$\bullet$};
        \draw [blue] (-0.4,-0.1) node {$\bullet$};

        \node (120b) at (-1,-1) [white] {$(1,2,0)$};
        \draw [black!70, thick] (-1.4,-1.1)--(-0.6,-1.1);
        \draw [black!70, thick] (-1.4,-1.1) arc (135:45:0.567);
        \draw [black] (-1,-1.1) node {$\bullet$};
        \draw [blue] (-0.6,-1.1) node {$\bullet$};
        \draw [red] (-1.4,-1.1) node {$\bullet$};

        \node (201b) at (1,-1) [white] {$(2,0,1)$};
        \draw [black!70, thick] (0.6,-1.1)--(1.4,-1.1);
        \draw [black!70, thick] (0.6,-1.1) arc (135:45:0.567);
        \draw [blue] (1,-1.1) node {$\bullet$};
        \draw [red] (1.4,-1.1) node {$\bullet$};
        \draw [black] (0.6,-1.1) node {$\bullet$};

        \draw [->] (012b)--(120b);
        \draw [->] (120b)--(201b);
        \draw [->] (201b)--(012b);
    \end{scope}
    \begin{scope}[shift={(3,-9)}]
        \node (021b) at (0,0) [white] {$(0,2,1)$};
        \draw [black!70, thick] (-0.4,-0.1)--(0.4,-0.1);
        \draw [black!70, thick] (-0.4,-0.1) arc (135:45:0.567);
        \draw [black] (0,-0.1) node {$\bullet$};
        \draw [red] (0.4,-0.1) node {$\bullet$};
        \draw [blue] (-0.4,-0.1) node {$\bullet$};

        \node (102b) at (-1,-1) [white] {$(1,0,2)$};
        \draw [black!70, thick] (-1.4,-1.1)--(-0.6,-1.1);
        \draw [black!70, thick] (-1.4,-1.1) arc (135:45:0.567);
        \draw [blue] (-1,-1.1) node {$\bullet$};
        \draw [black] (-0.6,-1.1) node {$\bullet$};
        \draw [red] (-1.4,-1.1) node {$\bullet$};

        \node (210b) at (1,-1) [white] {$(2,1,0)$};
        \draw [black!70, thick] (0.6,-1.1)--(1.4,-1.1);
        \draw [black!70, thick] (0.6,-1.1) arc (135:45:0.567);
        \draw [red] (1,-1.1) node {$\bullet$};
        \draw [blue] (1.4,-1.1) node {$\bullet$};
        \draw [black] (0.6,-1.1) node {$\bullet$};

        \draw [->] (021b)--(102b);
        \draw [->] (102b)--(210b);
        \draw [->] (210b)--(021b);
    \end{scope}
\end{tikzpicture}
\caption{\label{fig:stateK3}The $R$-state transition digraph of $K_{3}$, with $R=1$ and $R=2$. In each state, the healthy, infected and dead vertices are represented by nodes of colors blue, red and black, respectively.}
\end{figure}

\begin{theorem}
    Let $n\in \mathbb{Z}^{+}$. There exists a graph $G$, a positive integer $R$ and an admissible initial state $f_{0}$ of $G$ such that the cycle on the weakly connected component of $f_{0}$ has length $n$.
\end{theorem}

\begin{proof}
    For $n=1$ it is enough to consider any graph $G$ and the admissible initial state $\boldsymbol{0}_{G}$. For $n=2$ it is enough to consider the cycle graph on $4$ vertices $C_{4}$, $R=2$, and consider the admissible initial state where the vertices are alternately healthy and infected. For $n=3$, there is an example in the proof of Theorem \ref{thm:unbounded-gap}, when we considered the case $k+2\le R\le 2k$.

    For even $n\ge 4$, let $2k=n$, and let $G_{2k}$ be the graph with $k+3$ vertices $v_{0}, v'_{0},v_{1},v_{2},\dots,v_{k-1}, v_{k}, v'_{k}$. The edge set of $G_{2k}$ is given by $E=\{(v_{t},v_{t+1}):0\le t\le k-1\}\cup\{(v'_0,v_1),(v_{k-1},v'_{k})\}$.

\begin{figure}[ht]
\centering
\begin{tikzpicture}
    \draw (0,0.5) node {$\bullet$};
    \draw (0,-0.5) node {$\bullet$};
    \draw (1,0) node {$\bullet$};
    \draw (2,0) node {$\bullet$};
    \draw (3,0) node {$\bullet$};
    \draw (5,0) node {$\bullet$};
    \draw (6,0) node {$\bullet$};
    \draw (7,0.5) node {$\bullet$};
    \draw (7,-0.5) node {$\bullet$};
    \draw (0,0.5) [below] node {$v_0$};
    \draw (0,-0.5) [below] node {$v'_0$};
    \draw (1,0) [below] node {$v_1$};
    \draw (2,0) [below] node {$v_2$};
    \draw (3,0) [below] node {$v_3$};
    \draw (5,0) [below] node {$v_{k-2}$};
    \draw (6,0) [below] node {$v_{k-1}$};
    \draw (7,0.5) [below] node {$v_{k}$};
    \draw (7,-0.5) [below] node {$v'_{k}$};
    \draw (0,0.5)--(1,0)--(0,-0.5);
    \draw (7,0.5)--(6,0)--(7,-0.5);
    \draw (1,0)--(3.5,0);
    \draw (4,0) node {$\dots$};
    \draw (4.5,0)--(6,0);
\end{tikzpicture}
\caption{\label{fig:Gpar} The graph $G_{2k}$.}
\end{figure}
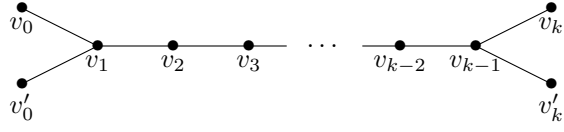

Figure \ref{fig:Gpar} shows the graph $G_{2k}$. Let $R=2$ and consider the admissible initial state of $G_{2k}$ where $\II_{0}=\{v_{0},v'_0\}$, and the remaining vertices are healthy. We will see that $f_{1}=f_{n+1}$. This follows from the fact that for every $1\le t\le 2k+1=n+1$:
\[\II_{t}=\begin{cases}\{v_0,v'_0\}&\text{if }t=2k,\\
\{v_k,v'_k\}&\text{if }t=k,\\
\{v_j\}&\text{if }t\in\{j,n-j\},1\le j\le k-1,\\
\{v_{1}\}&\text{if }t=2k+1.  
\end{cases}\qquad
\DD_{t}=\begin{cases}\{v_0,v'_0\}&\text{if }t\in\{1,2k+1\},\\
\{v_k,v'_k\}&\text{if }t=k+1,\\
\{v_j\}&\text{if }t\in\{j+1,n-j+1\}.\\
\end{cases}\]
Therefore, $f_{1}=f_{2k+1}$. Also, observe that the cycle has length $2k$ since $\II_{t}=\{v_{0},v'_{0}\}$ for exactly one value of $t\in \{1,2,\dots,2k+1\}$.

For odd $n\ge 5$, let $2k+1=n$, and let $G_{2k+1}$ be the graph with $2k+1$ vertices $v_{0},v_{1},\dots,v_{2k}$. The set of edges of $G_{2k+1}$ is given by $E=\{(v_{t},v_{t+1}):0\le t\le 2k-1\}\cup \{(v_{k-1},v_{k+1})\}$.

\begin{figure}[ht]
\centering
\begin{tikzpicture}
    \draw (0,0) node {$\bullet$};
    \draw (1,0) node {$\bullet$};
    \draw (3,0) node {$\bullet$};
    \draw (4,0) node {$\bullet$};
    \draw (5,0) node {$\bullet$};
    \draw (6,0) node {$\bullet$};
    \draw (8,0) node {$\bullet$};
    \draw (9,0) node {$\bullet$};
    \draw (4.5,0.7) node {$\bullet$};
    
    \draw (0,0) [below] node {$v_0$};
    \draw (1,0) [below] node {$v_1$};
    \draw (2,0) node {$\dots $};
    \draw (3,0) [below] node {$v_{k-2}$};
    \draw (4,0) [below] node {$v_{k-1}$};
    \draw (5,0) [below] node {$v_{k+1}$};
    \draw (6,0) [below] node {$v_{k+2}$};
    \draw (7,0) node {$\dots$};
    \draw (8,0) [below] node {$v_{2k-1}$};
    \draw (9,0) [below] node {$v_{2k}$};
    \draw (4.5,0.7) [above] node {$v_{k}$};
    
    \draw (0,0)--(1.5,0);
    \draw (2.5,0)--(6.5,0);
    \draw (7.5,0)--(9,0);
    \draw (4,0)--(4.5,0.7)--(5,0);
\end{tikzpicture}
\caption{\label{fig:Gimpar} The graph $G_{2k+1}$.}
\end{figure}
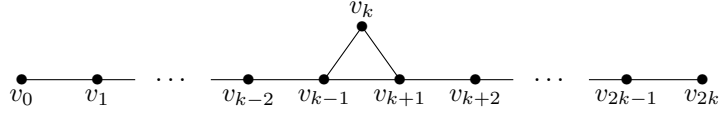
Figure \ref{fig:Gimpar} shows the graph $G_{2k+1}$. Now, let $R=2$ and consider the admissible initial state of $G_{2k+1}$ where $\II_{0}=\{v_{k-2},v_k\}$, and the remaining vertices are healthy. We will see that $f_{1}=f_{n+1}$. Indeed, if $n=5$, then: 
\begin{gather*}
\II_{0}=\DD_{1}=\{v_0,v_2\} \qquad \II_{1}=\DD_{2}=\{v_1,v_3\}\qquad \II_{2}=\DD_{3}=\{v_2,v_4\}\qquad \II_{3}=\DD_{4}=\{v_3\}\\
\II_{4}=\DD_{5}=\{v_1\} \qquad \II_{5}=\DD_{6}=\{v_0,v_2\}\qquad \II_{6}=\DD_{7}=\{v_1,v_3\}
\end{gather*}
and if $n=2k+1\ge 7$, $\DD_{1}=\{v_{k-2},v_{k}\}$, and for every $1\le t\le 2k+2=n+1$:
\[\II_t=\DD_{t+1}=\begin{cases}\{v_{k-t-2+2s}:0\le s\le t+1\}&\text{if }1\le t\le k-2,\\
\{v_1,v_3,v_5,\dots,v_{2k-1}\}&\text{if }t=k-1,\\
\{v_{t-k+2+2s}:0\le s\le 2k-t-1\}&\text{if }k\le t\le 2k-1,\\
\{v_{k-(t-n+2)+2s}:0\le s\le t-n+1\}&\text{if }n-1\le t\le n+1.
\end{cases}\]
Since $\II_{1}=\II_{2k+2}$ and $\DD_{1}=\DD_{2k+2}$, we have that $f_{1}=f_{n+1}$. Since the sets $\II_{t}$ are distinct for all $1\le t\le 2k+1$, we conclude that the length of the cycle of the weakly connected component of $f_{0}$ is $2k+1=n$.
\begin{figure}[ht]
    \centering
    \begin{tikzpicture}[scale=0.99]
    \draw (1.75,2) node {$G_{6}$};
    \foreach \x in {0,1,2,...,7}{
        \draw (-5.325+2*\x,1.25) node {$f_{\x}$};
    }
    \begin{scope}[shift={(-6,0)},scale=0.8]
        \draw (0,0.5)--(0.5,0)--(1,0)--(1.5,0.5);
        \draw (0,-0.5)--(0.5,0);
        \draw (1,0)--(1.5,-0.5);
        \draw (0,0.5) [above] node {$v_0$};
        \draw (0,-0.5) [below] node {$v'_0$};
        \draw (0.5,0) [left] node {$v_1$};
        \draw (1,0) [right] node {$v_2$};
        \draw (1.5,0.5) [above] node {$v_3$};
        \draw (1.5,-0.5) [below] node {$v'_3$};
        \draw (0,-0.5) [red] node {$\bullet$};
        \draw (0,0.5) [red] node {$\bullet$};
        \draw (0.5,0) [blue] node {$\bullet$};
        \draw (1,0) [blue] node {$\bullet$};
        \draw (1.5,-0.5) [blue] node {$\bullet$};
        \draw (1.5,0.5) [blue] node {$\bullet$};
    \end{scope}
    \begin{scope}[shift={(-4,0)},scale=0.8]
        \draw (0,0.5)--(0.5,0)--(1,0)--(1.5,0.5);
        \draw (0,-0.5)--(0.5,0);
        \draw (1,0)--(1.5,-0.5);
        \draw (0,0.5) [above] node {$v_0$};
        \draw (0,-0.5) [below] node {$v'_0$};
        \draw (0.5,0) [left] node {$v_1$};
        \draw (1,0) [right] node {$v_2$};
        \draw (1.5,0.5) [above] node {$v_3$};
        \draw (1.5,-0.5) [below] node {$v'_3$};
        \draw (0,-0.5) [black] node {$\bullet$};
        \draw (0,0.5) [black] node {$\bullet$};
        \draw (0.5,0) [red] node {$\bullet$};
        \draw (1,0) [blue] node {$\bullet$};
        \draw (1.5,-0.5) [blue] node {$\bullet$};
        \draw (1.5,0.5) [blue] node {$\bullet$};
    \end{scope}
    \begin{scope}[shift={(-2,0)},scale=0.8]
        \draw (0,0.5)--(0.5,0)--(1,0)--(1.5,0.5);
        \draw (0,-0.5)--(0.5,0);
        \draw (1,0)--(1.5,-0.5);
        \draw (0,0.5) [above] node {$v_0$};
        \draw (0,-0.5) [below] node {$v'_0$};
        \draw (0.5,0) [left] node {$v_1$};
        \draw (1,0) [right] node {$v_2$};
        \draw (1.5,0.5) [above] node {$v_3$};
        \draw (1.5,-0.5) [below] node {$v'_3$};
        \draw (0,-0.5) [blue] node {$\bullet$};
        \draw (0,0.5) [blue] node {$\bullet$};
        \draw (0.5,0) [black] node {$\bullet$};
        \draw (1,0) [red] node {$\bullet$};
        \draw (1.5,-0.5) [blue] node {$\bullet$};
        \draw (1.5,0.5) [blue] node {$\bullet$};
    \end{scope}
    \begin{scope}[shift={(0,0)},scale=0.8]
        \draw (0,0.5)--(0.5,0)--(1,0)--(1.5,0.5);
        \draw (0,-0.5)--(0.5,0);
        \draw (1,0)--(1.5,-0.5);
        \draw (0,0.5) [above] node {$v_0$};
        \draw (0,-0.5) [below] node {$v'_0$};
        \draw (0.5,0) [left] node {$v_1$};
        \draw (1,0) [right] node {$v_2$};
        \draw (1.5,0.5) [above] node {$v_3$};
        \draw (1.5,-0.5) [below] node {$v'_3$};
        \draw (0,-0.5) [blue] node {$\bullet$};
        \draw (0,0.5) [blue] node {$\bullet$};
        \draw (0.5,0) [blue] node {$\bullet$};
        \draw (1,0) [black] node {$\bullet$};
        \draw (1.5,-0.5) [red] node {$\bullet$};
        \draw (1.5,0.5) [red] node {$\bullet$};
    \end{scope}
    \begin{scope}[shift={(2,0)},scale=0.8]
        \draw (0,0.5)--(0.5,0)--(1,0)--(1.5,0.5);
        \draw (0,-0.5)--(0.5,0);
        \draw (1,0)--(1.5,-0.5);
        \draw (0,0.5) [above] node {$v_0$};
        \draw (0,-0.5) [below] node {$v'_0$};
        \draw (0.5,0) [left] node {$v_1$};
        \draw (1,0) [right] node {$v_2$};
        \draw (1.5,0.5) [above] node {$v_3$};
        \draw (1.5,-0.5) [below] node {$v'_3$};
        \draw (0,-0.5) [blue] node {$\bullet$};
        \draw (0,0.5) [blue] node {$\bullet$};
        \draw (0.5,0) [blue] node {$\bullet$};
        \draw (1,0) [red] node {$\bullet$};
        \draw (1.5,-0.5) [black] node {$\bullet$};
        \draw (1.5,0.5) [black] node {$\bullet$};
    \end{scope}
    \begin{scope}[shift={(4,0)},scale=0.8]
        \draw (0,0.5)--(0.5,0)--(1,0)--(1.5,0.5);
        \draw (0,-0.5)--(0.5,0);
        \draw (1,0)--(1.5,-0.5);
        \draw (0,0.5) [above] node {$v_0$};
        \draw (0,-0.5) [below] node {$v'_0$};
        \draw (0.5,0) [left] node {$v_1$};
        \draw (1,0) [right] node {$v_2$};
        \draw (1.5,0.5) [above] node {$v_3$};
        \draw (1.5,-0.5) [below] node {$v'_3$};
        \draw (0,-0.5) [blue] node {$\bullet$};
        \draw (0,0.5) [blue] node {$\bullet$};
        \draw (0.5,0) [red] node {$\bullet$};
        \draw (1,0) [black] node {$\bullet$};
        \draw (1.5,-0.5) [blue] node {$\bullet$};
        \draw (1.5,0.5) [blue] node {$\bullet$};
    \end{scope}
    \begin{scope}[shift={(6,0)},scale=0.8]
        \draw (0,0.5)--(0.5,0)--(1,0)--(1.5,0.5);
        \draw (0,-0.5)--(0.5,0);
        \draw (1,0)--(1.5,-0.5);
        \draw (0,0.5) [above] node {$v_0$};
        \draw (0,-0.5) [below] node {$v'_0$};
        \draw (0.5,0) [left] node {$v_1$};
        \draw (1,0) [right] node {$v_2$};
        \draw (1.5,0.5) [above] node {$v_3$};
        \draw (1.5,-0.5) [below] node {$v'_3$};
        \draw (0,-0.5) [red] node {$\bullet$};
        \draw (0,0.5) [red] node {$\bullet$};
        \draw (0.5,0) [black] node {$\bullet$};
        \draw (1,0) [blue] node {$\bullet$};
        \draw (1.5,-0.5) [blue] node {$\bullet$};
        \draw (1.5,0.5) [blue] node {$\bullet$};
    \end{scope}
    \begin{scope}[shift={(8,0)},scale=0.8]
        \draw (0,0.5)--(0.5,0)--(1,0)--(1.5,0.5);
        \draw (0,-0.5)--(0.5,0);
        \draw (1,0)--(1.5,-0.5);
        \draw (0,0.5) [above] node {$v_0$};
        \draw (0,-0.5) [below] node {$v'_0$};
        \draw (0.5,0) [left] node {$v_1$};
        \draw (1,0) [right] node {$v_2$};
        \draw (1.5,0.5) [above] node {$v_3$};
        \draw (1.5,-0.5) [below] node {$v'_3$};
        \draw (0,-0.5) [black] node {$\bullet$};
        \draw (0,0.5) [black] node {$\bullet$};
        \draw (0.5,0) [red] node {$\bullet$};
        \draw (1,0) [blue] node {$\bullet$};
        \draw (1.5,-0.5) [blue] node {$\bullet$};
        \draw (1.5,0.5) [blue] node {$\bullet$};
    \end{scope}

    \draw (-6.25,-1.25)--(9.5,-1.25) [dashed];
    \draw (1.75,-1.75) node {$G_{7}$};
    \begin{scope}[shift={(-6,-3.75)}]
    \draw (1.25,1.5) node {$f_0$};
    \draw (0,0)--(2.5,0);
    \draw (1,0)--(1.25,0.5)--(1.5,0);
    \draw (0,0) node [below] {$v_{0}$};
    \draw (0.5,0) node [below] {$v_{1}$};
    \draw (1,0) node [below] {$v_{2}$};
    \draw (1.25,0.5) node [above] {$v_{3}$};
    \draw (1.5,0) node [below] {$v_{4}$};
    \draw (2,0) node [below] {$v_{5}$};
    \draw (2.5,0) node [below] {$v_{6}$};
    \draw (0,0) node [blue] {$\bullet$};
    \draw (0.5,0) node [red] {$\bullet$};
    \draw (1,0) node [blue] {$\bullet$};
    \draw (1.25,0.5) node [red] {$\bullet$};
    \draw (1.5,0) node [blue] {$\bullet$};
    \draw (2,0) node [blue] {$\bullet$};
    \draw (2.5,0) node [blue] {$\bullet$};
    \end{scope}
    \begin{scope}[shift={(-2.75,-3.75)}]
    \draw (1.25,1.5) node {$f_1$};
    \draw (0,0)--(2.5,0);
    \draw (1,0)--(1.25,0.5)--(1.5,0);
    \draw (0,0) node [below] {$v_{0}$};
    \draw (0.5,0) node [below] {$v_{1}$};
    \draw (1,0) node [below] {$v_{2}$};
    \draw (1.25,0.5) node [above] {$v_{3}$};
    \draw (1.5,0) node [below] {$v_{4}$};
    \draw (2,0) node [below] {$v_{5}$};
    \draw (2.5,0) node [below] {$v_{6}$};
    \draw (0,0) node [red] {$\bullet$};
    \draw (0.5,0) node [black] {$\bullet$};
    \draw (1,0) node [red] {$\bullet$};
    \draw (1.25,0.5) node [black] {$\bullet$};
    \draw (1.5,0) node [red] {$\bullet$};
    \draw (2,0) node [blue] {$\bullet$};
    \draw (2.5,0) node [blue] {$\bullet$};
    \end{scope}
    \begin{scope}[shift={(0.5,-3.75)}]
    \draw (1.25,1.5) node {$f_2$};
    \draw (0,0)--(2.5,0);
    \draw (1,0)--(1.25,0.5)--(1.5,0);
    \draw (0,0) node [below] {$v_{0}$};
    \draw (0.5,0) node [below] {$v_{1}$};
    \draw (1,0) node [below] {$v_{2}$};
    \draw (1.25,0.5) node [above] {$v_{3}$};
    \draw (1.5,0) node [below] {$v_{4}$};
    \draw (2,0) node [below] {$v_{5}$};
    \draw (2.5,0) node [below] {$v_{6}$};
    \draw (0,0) node [black] {$\bullet$};
    \draw (0.5,0) node [red] {$\bullet$};
    \draw (1,0) node [black] {$\bullet$};
    \draw (1.25,0.5) node [red] {$\bullet$};
    \draw (1.5,0) node [black] {$\bullet$};
    \draw (2,0) node [red] {$\bullet$};
    \draw (2.5,0) node [blue] {$\bullet$};
    \end{scope}
    \begin{scope}[shift={(3.75,-3.75)}]
    \draw (1.25,1.5) node {$f_3$};
    \draw (0,0)--(2.5,0);
    \draw (1,0)--(1.25,0.5)--(1.5,0);
    \draw (0,0) node [below] {$v_{0}$};
    \draw (0.5,0) node [below] {$v_{1}$};
    \draw (1,0) node [below] {$v_{2}$};
    \draw (1.25,0.5) node [above] {$v_{3}$};
    \draw (1.5,0) node [below] {$v_{4}$};
    \draw (2,0) node [below] {$v_{5}$};
    \draw (2.5,0) node [below] {$v_{6}$};
    \draw (0,0) node [blue] {$\bullet$};
    \draw (0.5,0) node [black] {$\bullet$};
    \draw (1,0) node [red] {$\bullet$};
    \draw (1.25,0.5) node [black] {$\bullet$};
    \draw (1.5,0) node [red] {$\bullet$};
    \draw (2,0) node [black] {$\bullet$};
    \draw (2.5,0) node [red] {$\bullet$};
    \end{scope}
    \begin{scope}[shift={(7,-3.75)}]
    \draw (1.25,1.5) node {$f_4$};
    \draw (0,0)--(2.5,0);
    \draw (1,0)--(1.25,0.5)--(1.5,0);
    \draw (0,0) node [below] {$v_{0}$};
    \draw (0.5,0) node [below] {$v_{1}$};
    \draw (1,0) node [below] {$v_{2}$};
    \draw (1.25,0.5) node [above] {$v_{3}$};
    \draw (1.5,0) node [below] {$v_{4}$};
    \draw (2,0) node [below] {$v_{5}$};
    \draw (2.5,0) node [below] {$v_{6}$};
    \draw (0,0) node [blue] {$\bullet$};
    \draw (0.5,0) node [blue] {$\bullet$};
    \draw (1,0) node [black] {$\bullet$};
    \draw (1.25,0.5) node [red] {$\bullet$};
    \draw (1.5,0) node [black] {$\bullet$};
    \draw (2,0) node [red] {$\bullet$};
    \draw (2.5,0) node [black] {$\bullet$};
    \end{scope}
    \begin{scope}[shift={(-4.75,-6.5)}]
    \draw (1.25,1.5) node {$f_5$};
    \draw (0,0)--(2.5,0);
    \draw (1,0)--(1.25,0.5)--(1.5,0);
    \draw (0,0) node [below] {$v_{0}$};
    \draw (0.5,0) node [below] {$v_{1}$};
    \draw (1,0) node [below] {$v_{2}$};
    \draw (1.25,0.5) node [above] {$v_{3}$};
    \draw (1.5,0) node [below] {$v_{4}$};
    \draw (2,0) node [below] {$v_{5}$};
    \draw (2.5,0) node [below] {$v_{6}$};
    \draw (0,0) node [blue] {$\bullet$};
    \draw (0.5,0) node [blue] {$\bullet$};
    \draw (1,0) node [blue] {$\bullet$};
    \draw (1.25,0.5) node [black] {$\bullet$};
    \draw (1.5,0) node [red] {$\bullet$};
    \draw (2,0) node [black] {$\bullet$};
    \draw (2.5,0) node [blue] {$\bullet$};
    \end{scope}
    \begin{scope}[shift={(-1.25,-6.5)}]
    \draw (1.25,1.5) node {$f_6$};
    \draw (0,0)--(2.5,0);
    \draw (1,0)--(1.25,0.5)--(1.5,0);
    \draw (0,0) node [below] {$v_{0}$};
    \draw (0.5,0) node [below] {$v_{1}$};
    \draw (1,0) node [below] {$v_{2}$};
    \draw (1.25,0.5) node [above] {$v_{3}$};
    \draw (1.5,0) node [below] {$v_{4}$};
    \draw (2,0) node [below] {$v_{5}$};
    \draw (2.5,0) node [below] {$v_{6}$};
    \draw (0,0) node [blue] {$\bullet$};
    \draw (0.5,0) node [blue] {$\bullet$};
    \draw (1,0) node [red] {$\bullet$};
    \draw (1.25,0.5) node [blue] {$\bullet$};
    \draw (1.5,0) node [black] {$\bullet$};
    \draw (2,0) node [blue] {$\bullet$};
    \draw (2.5,0) node [blue] {$\bullet$};
    \end{scope}
    \begin{scope}[shift={(2.25,-6.5)}]
    \draw (1.25,1.5) node {$f_7$};
    \draw (0,0)--(2.5,0);
    \draw (1,0)--(1.25,0.5)--(1.5,0);
    \draw (0,0) node [below] {$v_{0}$};
    \draw (0.5,0) node [below] {$v_{1}$};
    \draw (1,0) node [below] {$v_{2}$};
    \draw (1.25,0.5) node [above] {$v_{3}$};
    \draw (1.5,0) node [below] {$v_{4}$};
    \draw (2,0) node [below] {$v_{5}$};
    \draw (2.5,0) node [below] {$v_{6}$};
    \draw (0,0) node [blue] {$\bullet$};
    \draw (0.5,0) node [red] {$\bullet$};
    \draw (1,0) node [black] {$\bullet$};
    \draw (1.25,0.5) node [red] {$\bullet$};
    \draw (1.5,0) node [blue] {$\bullet$};
    \draw (2,0) node [blue] {$\bullet$};
    \draw (2.5,0) node [blue] {$\bullet$};
    \end{scope}
    \begin{scope}[shift={(5.75,-6.5)}]
    \draw (1.25,1.5) node {$f_8=f_1$};
    \draw (0,0)--(2.5,0);
    \draw (1,0)--(1.25,0.5)--(1.5,0);
    \draw (0,0) node [below] {$v_{0}$};
    \draw (0.5,0) node [below] {$v_{1}$};
    \draw (1,0) node [below] {$v_{2}$};
    \draw (1.25,0.5) node [above] {$v_{3}$};
    \draw (1.5,0) node [below] {$v_{4}$};
    \draw (2,0) node [below] {$v_{5}$};
    \draw (2.5,0) node [below] {$v_{6}$};
    \draw (0,0) node [red] {$\bullet$};
    \draw (0.5,0) node [black] {$\bullet$};
    \draw (1,0) node [red] {$\bullet$};
    \draw (1.25,0.5) node [black] {$\bullet$};
    \draw (1.5,0) node [red] {$\bullet$};
    \draw (2,0) node [blue] {$\bullet$};
    \draw (2.5,0) node [blue] {$\bullet$};
    \end{scope}
\end{tikzpicture}
    \caption{\label{fig:Gevo} The evolution of $G_{6}$ and $G_{7}$ with $R=2$. Blue, red and black vertices represent the healthy, infected and dead vertices at each time.}
\end{figure}
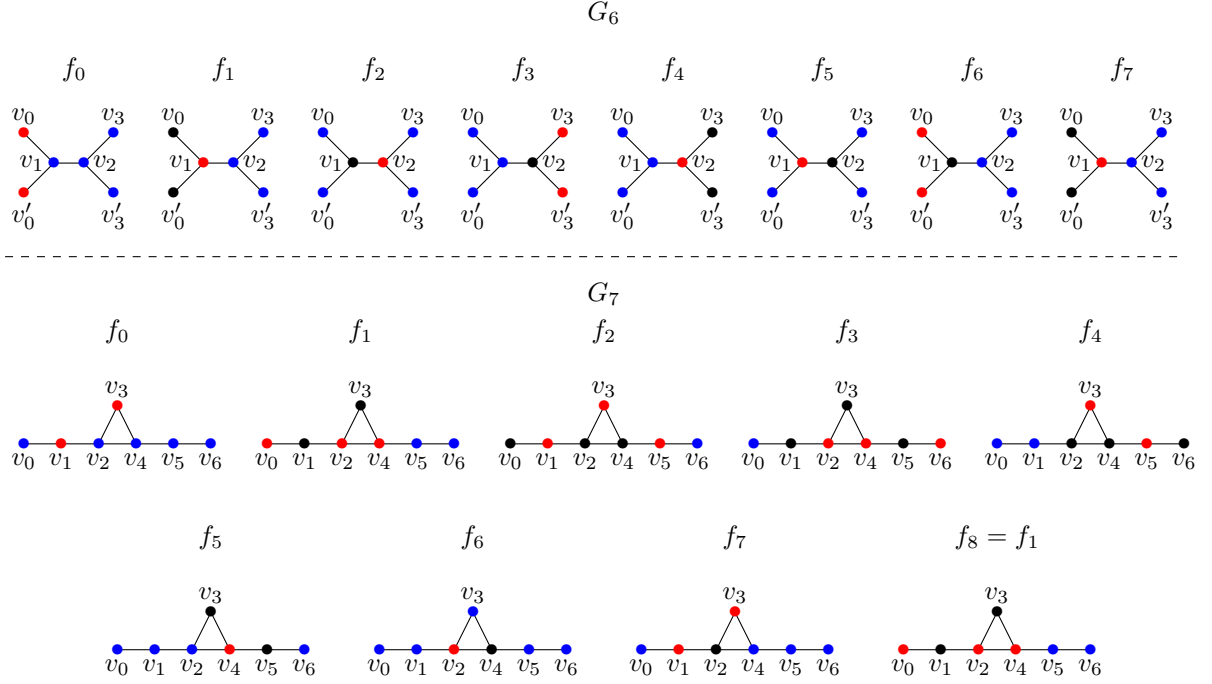
Figure \ref{fig:Gevo} shows this evolution for $G_{6}$ and $G_{7}$ with the initial state defined above.
\end{proof}

The following theorem characterizes the directed cycles of length $2$ in every $R$-state transition digraph.

\begin{theorem}\label{thm:2cyles}
    Let $G$ be a graph and let $R\in\mathbb{Z}^{+}$. The $R$-state transition digraph contains a cycle of length $2$ if and only if there exists a partition of the vertices $V=U_0\cup U_1$ such that each vertex in $U_{0}$ has at least $R$ neighbors in $U_1$ and each vertex in $U_{1}$ has at least $R$ neighbors in $U_{0}$.
\end{theorem}

\begin{proof}
    Let $G$ be a graph and assume there exists a partition of the vertices $V=U_{0}\cup U_{1}$ such that each vertex in $U_{0}$ has at least $R$ neighbors in $U_1$ and each vertex in $U_{1}$ has at least $R$ neighbors in $U_{0}$. In this case, taking the admissible initial state $f_{0}$ for which $U_{0}=\HH_{0}$ and $U_{1}=\II_{0}$, the evolution is given, for every $t\geq 1$,
    by $\HH_{t}=\emptyset$,
    \[\II_{t}=\begin{cases}
        U_1&\text{if }t\text{ is odd,}\\
        U_0&\text{if }t\text{ is even,}
    \end{cases}
    \qquad\text{and}\qquad
    \DD_{t}=\begin{cases}
        U_0&\text{if }t\text{ is odd,}\\
        U_1&\text{if }t\text{ is even.}
    \end{cases}\]
    Then, this initial state enters a cycle of length $2$ in the $R$-state transition digraph of $G$.

    For the converse, we proceed by contradiction. Assume there exists a cycle of length $2$ in the $R$-state transition digraph of $G$, but there does not exist a partition of the vertices $V=U_{0}\cup U_{1}$ such that each vertex in $U_{0}$ has at least $R$ neighbors in $U_1$ and each vertex in $U_{1}$ has at least $R$ neighbors in $U_{0}$. Let $f$ and $g$ be the states in this cycle.

    Since the cycle has length $2$, each dead vertex in the cycle must become infected in order to be dead again two steps later.

    First, notice that if a vertex is healthy in one of the states, then it must be healthy in both states. Indeed, if $f(v)=0$ and $g(v)=1$, then $g(v)=1$ would imply $f(v)=2$, since every infected vertex becomes dead in the next time step, which is a contradiction. Analogously, $g(v)=0$ implies $f(v)=0$.

    Let $W\subseteq V$ be the set of vertices that are infected or dead in either $f$ or $g$. Assume that there exists a vertex that is healthy in $f$ or $g$. Since the cycle has length $2$, the set $W$ is nonempty. Moreover, since $G$ is connected, there exists an edge $v_0w$ such that $v_0$ is healthy in both $f$ and $g$, and $w\in W$. The vertex $w$ is infected in exactly one of the states $f$ and $g$. Therefore, $v_0$ has an infected neighbor in one of these states and must become infected in the other state, contradicting the fact that it is healthy in both. Then, $f^{-1}(0)=g^{-1}(0)=\emptyset.$

    In this case, $f^{-1}(2)\cup f^{-1}(1)$ is a partition of the vertices, and  $f^{-1}(2)=g^{-1}(1)$ and $f^{-1}(1)=g^{-1}(2)$. Let $U_0=f^{-1}(1)$ and $U_1=f^{-1}(2)$. If $v\in U_1=f^{-1}(2)$, then $v$ evolves from dead in $f$ to infected in $g$. Hence, $v$ has at least $R$ infected neighbors in the state $f$, that is, at least $R$ neighbors in $f^{-1}(1)=U_0$. Similarly, if $v\in U_0=f^{-1}(1)$, then $v\in g^{-1}(2)$ and evolves from dead in $g$ to infected in $f$. Hence, $v$ has at least $R$ infected neighbors in the state $g$, that is, at least $R$ neighbors in $g^{-1}(1)=U_1$.

    Therefore, $V=U_0\cup U_1$ is a partition such that each vertex in $U_0$ has at least $R$ neighbors in $U_1$, and each vertex in $U_1$ has at least $R$ neighbors in $U_0$, contradicting our assumption.
\end{proof}

It is worth noting that this theorem provides a method for finding cycles of length $2$ in $D_R(G)$. In fact, if we can split $V(G)$ into two sets $U_0$ and $U_1$ as in Theorem \ref{thm:2cyles}, then we only need to compute $(f_t^{R,f_0})_{t\geq 0}$ to obtain the corresponding cycle.

\section{Results for families of graphs}\label{sec:graphFamilies}

We determine the extinction parameters $\hiv(G)$ and $\HIV(G)$ for several well-known graph families, providing explicit examples of the general theory. In particular, we derive closed formulas for cycles, complete graphs, and complete bipartite graphs, showing that the values for the extinction parameters depend on the complete structure of the graph.

\begin{theorem}
    For every integer $n\ge 3$, 
    \[\HIV(C_{n})=\hiv(C_{n})=\begin{cases}
        3 & \text{if $n$ is even,}\\
        2 & \text{if $n$ is odd.}
    \end{cases}\]
\end{theorem}

\begin{proof}
If $n$ is even, the graph is a bipartite regular graph. From Theorem \ref{thm:bipartiteregular} we have that $\HIV(C_{n})=\hiv(C_{n})=\Delta+1=3$. If $n$ is odd, the graph is not bipartite, and the Theorem \ref{thm:bipartiteregular} implies that $\HIV(C_{n})<\Delta+1=3$. From this and Theorem \ref{thm:bound} we have that $\HIV(C_{n})=\hiv(C_{n})=2$ when $n$ is odd.
\end{proof}

\begin{theorem}
    For every integer $n\geq 3$, $\HIV(K_{n})=\hiv(K_{n}) =\left\lceil\frac{n+1}{2}\right\rceil.$
\end{theorem}

\begin{proof}
    Let $R\geq \left\lceil\tfrac{n+1}{2}\right\rceil$ and let $f_{0}$ be any initial state for $K_{n}$. If the initial state has no infected vertices, then $R$ guarantees the extinction of the infection. If the initial state has at least one infected vertex, let $A=\HH_{0}$ and $B=\II_{0}$. Notice that, if $|A|\geq R$, then $|B|=n-|A|\leq n-R<R,$  since $2R\geq n+1$. We have that the infection evolves the vertices in $A$ and $B$ as follows, depending on the size of $A$:
    \begin{center}
    $\begin{tblr}{colspec={*{19}{c}}, colsep=1pt,rowsep=0pt}
    &\qquad&
    \SetCell[c=7]{c} \text{if }\vert A\vert < R &&&&&&&\qquad&
    \SetCell[c=9]{c}\text{if }\vert A\vert \geq R&&&&&&&&&\\
    && f_{0}&&f_{1}&&f_{2}&&f_{3}& &f_{0}&&f_{1}&&f_{2}&&f_{3}&&f_{4}\\\hline
    A&&\HH_{0}&\to&\II_{1}&\to&\DD_{2}&\to&\HH_{3}& &\HH_{0}&\to&\II_{1}&\to&\DD_{2}&\to&\HH_{3}&\to&\HH_{4}\\
    B&&\II_{0}&\to&\DD_{1}&\to&\HH_{2}&\to&\HH_{3}& &\II_{0}&\to&\DD_{1}&\to&\II_{2}&\to&\DD_{3}&\to&\HH_{4}
    \end{tblr}$
    \end{center}
    and then $R$ guarantees the extinction of the infection. It follows that $\HIV(K_{n}) \leq \left\lceil\tfrac{n+1}{2}\right\rceil.$

    On the other hand, if $R<\left\lceil\tfrac{n+1}{2}\right\rceil$, we can consider the initial state where there are exactly $\left\lceil\tfrac{n}{2}\right\rceil$ healthy vertices and $\left\lfloor\tfrac{n}{2}\right\rfloor$ infected vertices. Call $A$ and $B$ the sets of vertices in $\HH_{0}$ and $\II_{0}$, respectively. Since $R$ is an integer, $R\leq \left\lceil\frac{n+1}{2}\right\rceil-1=\left\lfloor\frac{n}{2}\right\rfloor$. Hence $|A|\geq R$ and $|B|\geq R$. Therefore, for any positive integer $t$, we have that
    \[\II_{t}=\begin{cases}
        A, & \text{if $t$ is odd,}\\
        B, & \text{if $t$ is even,}
    \end{cases}
    \qquad \text{and}\qquad
    \DD_{t}=\begin{cases}
        B, & \text{if $t$ is odd,}\\
        A, & \text{if $t$ is even,}
    \end{cases}\]
    and we conclude that $\hiv(K_{n}) \geq \left\lceil\tfrac{n+1}{2}\right\rceil$. From $\HIV(K_{n}) \leq \left\lceil\tfrac{n+1}{2}\right\rceil,$ $\hiv(K_{n})\geq\left\lceil\tfrac{n+1}{2}\right\rceil,$ and Theorem~\ref{thm:bound}, we conclude the proof.
\end{proof}

\begin{theorem}
For all positive integers $m\leq n$, $\HIV(K_{m,n})=\hiv(K_{m,n})
=\max\left\{m+1,\left\lfloor \frac{n}{2}\right\rfloor+1\right\}.$
\end{theorem}

\begin{proof}
Let $M$ and $N$ be the parts of $K_{m,n}$  of cardinalities $m$ and $n$, respectively, where $m\leq n$. For an initial state $f_0$, we will denote by $M_{\HH}$ and $M_{\II}$ the sets of healthy and infected vertices of $M$ at time step $0$, respectively. Analogously, we use the notation $N_{\HH}$ and $N_{\II}$.

If $m=n$, from Theorem \ref{thm:bipartiteregular} we have that  $\HIV(K_{m,n})=\hiv(K_{m,n})=m+1.$ In what follows, we assume that $m<n$.

From Theorem~\textcolor{blue}{\ref{thm:bipartiteregular}}, we have that $\HIV(K_{m,n})\leq \Delta(K_{m,n})=n,$ which implies that the infection becomes extinct for every $R\geq n$. First, notice that every $R\leq m$ is latent. Indeed, consider the initial state $M=M_{\II}$ and $N=N_{\HH}$. Then, for $t\geq1$, the infection evolves as
\[\II_{t}=\begin{cases}
    N, & \text{if $t$ is odd,}\\
    M, & \text{if $t$ is even,}
\end{cases}
\qquad\text{and}\qquad
\DD_{t}=\begin{cases}
    M, & \text{if $t$ is odd,}\\
    N, & \text{if $t$ is even.}
\end{cases}\]

Now, if $m<R\leq\left\lfloor \frac{n}{2}\right\rfloor$, then $R$ is latent. To prove this we use the admissible initial $M=M_{\HH}$, $|N_{\HH}|=\left\lfloor \frac{n}{2}\right\rfloor$ and $|N_{\II}|=\left\lceil \frac{n}{2}\right\rceil$. With this configuration, the states for $t\geq1$ are established as follows:
\[\HH_{t}=\begin{cases}
    N_{\HH}, & \text{if $t=4k$ or $t=4k+1$},\\
    N_{\II}, & \text{if $t=4k+2$ or $t=4k+3$},
\end{cases}\]
\[\II_{t}=\begin{cases}
    M, & \text{if $t$ is odd,}\\
    N_{\II}, & \text{if $t=4k$},\\
    N_{\HH}, & \text{if $t=4k+2$},
\end{cases}\qquad\text{and}\qquad\DD_{t}=\begin{cases}
    M, & \text{if $t$ is even,}\\
    N_{\II}, & \text{if $t=4k+1$},\\
    N_{\HH}, & \text{if $t=4k+3$}.
\end{cases}\]
Thus, every $R\leq\max\left\{m,\left\lfloor\frac{n}{2}\right\rfloor\right\}$ is latent. Therefore, $\hiv(K_{m,n})\geq\max\left\{m+1,\left\lfloor\frac{n}{2}\right\rfloor+1\right\}.$

It only remains to check the case when $m<R<n$ and $\left\lfloor\frac{n}{2}\right\rfloor<R$. Notice that the second inequality implies that $2R>n$. Therefore, $N_{\HH}$ and $N_{\II}$ cannot both have cardinality at least $R$. First, we consider the case when $|N_{\HH}|<R$ and $|N_{\II}|<R$, which implies that $N_{\HH},N_{\II}\neq\emptyset$. 

If $M_{\II}\neq\emptyset$, then $M_{\II}, N_{\II}\subseteq\DD_{1}\cap\CC_{2}\cap\CC_{3}$ and $M_{\HH},N_{\HH}\subseteq \II_{1}\cap\DD_{2}\cap\CC_{3}$. Therefore, the system reaches the all-healthy state at time $3$. If $M_{\II}=\emptyset$, then $M=M_{\HH}$ and the evolution is
\[\begin{tblr}{colspec=ccccccccccc,colsep=1pt,rowsep=0pt}
&& f_{0}&&f_{1}&&f_{2}&&f_{3}&&f_{4}\\\hline
M_{\HH}&\quad&\HH_{0}&\to&\II_{1}&\to&\DD_{2}&\to&\HH_{3}&\to&\HH_{4},\\
N_{\HH}&\quad&\HH_{0}&\to&\HH_{1}&\to&\II_{2}&\to&\DD_{3}&\to&\HH_{4},\\
N_{\II}&\quad&\II_{0}&\to&\DD_{1}&\to&\HH_{2}&\to&\HH_{3}&\to&\HH_{4}.
\end{tblr}\]
Hence, in this case, the system reaches the all-healthy state at time $4$. Now, for the case when $\max\{|N_{\HH}|,|N_{\II}|\}\geq R$, we have to analyze the following cases:
\begin{itemize}
\item If none of the sets $M_{\HH},M_{\II},N_{\HH},N_{\II}$ is empty, we have that
\begin{center}
$\begin{tblr}{colspec={*{29}{c}},colsep=1pt,rowsep=0pt}
&&\SetCell[c=7]{c}\text{if }|N_{\HH}|<R\leq |N_{\II}|&&&&&&&\qquad&
\SetCell[c=19]{c}\text{if }|N_{\II}|<R\leq |N_{\HH}|&&&&\\
&&f_{0}&&f_{1}&&f_{2}&&f_{3}&&f_{0}&&f_{1}&&f_{2}&&f_{3}&&f_{4}&&f_{5}&&f_{6}&&f_{7}&&f_{8}&&f_{9}\\\hline
M_{\HH}&\qquad&\HH_{0}&\to&\II_{1}&\to&\DD_{2}&\to&\HH_{3}&&
\HH_{0}&\to&\II_{1}&\to&\DD_{2}&\to&\HH_{3}&\to&\II_{4}&\to&\DD_{5}&\to&\II_{6}&\to&\DD_{7}&\to&\HH_{8}&\to&\HH_{9}\\
M_{\II}&\qquad&\II_{0}&\to&\DD_{1}&\to&\HH_{2}&\to&\HH_{3}&&
\II_{0}&\to&\DD_{1}&\to&\II_{2}&\to&\DD_{3}&\to&\HH_{4}&\to&\HH_{5}&\to&\II_{6}&\to&\DD_{7}&\to&\HH_{8}&\to&\HH_{9}\\
N_{\HH}&\qquad&\HH_{0}&\to&\II_{1}&\to&\DD_{2}&\to&\HH_{3}&&
\HH_{0}&\to&\II_{1}&\to&\DD_{2}&\to&\HH_{3}&\to&\HH_{4}&\to&\II_{5}&\to&\DD_{6}&\to&\HH_{7}&\to&\HH_{8}&\to&\HH_{9}\\
N_{\II}&\qquad&\II_{0}&\to&\DD_{1}&\to&\HH_{2}&\to&\HH_{3}&&
\II_{0}&\to&\DD_{1}&\to&\HH_{2}&\to&\II_{3}&\to&\DD_{4}&\to&\HH_{5}&\to&\HH_{6}&\to&\II_{7}&\to&\DD_{8}&\to&\HH_{9}
\end{tblr}$
\end{center}

\item If at least one of the sets $M_{\HH},M_{\II},N_{\HH},N_{\II}$ is empty, then a direct application of the transition rules gives the following bounds:
\[\begin{array}{c|c|c}
\text{Relation between the sets in $N$}&\text{Additional condition}&\text{Extinction time}\\ 
\hline|N_{\HH}|<R\leq |N_{\II}|&M_{\II}\neq\emptyset&\text{at most }3\\
|N_{\HH}|<R\leq |N_{\II}|&M_{\II}=\emptyset&\text{at most }4\\
|N_{\II}|<R\leq |N_{\HH}|&N_{\II}=\emptyset&\text{at most }4\\
|N_{\II}|<R\leq |N_{\HH}|&N_{\II}\neq\emptyset,\ M_{\HH}=\emptyset&\text{at most }5\\
|N_{\II}|<R\leq |N_{\HH}|&N_{\II}\neq\emptyset,\ M_{\II}=\emptyset&\text{at most }6
\end{array}\]
\end{itemize}

Thus, the extinction of the infection occurs in every case. We conclude that the infection becomes extinct when $\max\left\{m+1,\left\lfloor\frac{n}{2}\right\rfloor+1\right\}\leq R<n$. Since every $R\geq n$ also guarantees the extinction of the infection, it follows that $\HIV(K_{m,n})\leq\max\left\{m+1,\left\lfloor\frac{n}{2}\right\rfloor+1\right\}$.

Combining this inequality with the lower bound on $\hiv(K_{m,n})$ and the Theorem~\ref{thm:bound}, we obtain $\HIV(K_{m,n})=\hiv(K_{m,n})=\max\left\{m+1,\left\lfloor\frac{n}{2}\right\rfloor+1\right\}$.
\end{proof}

For wheel graphs $W_n=K_1\vee C_{n-1}$ of small size, we performed an exhaustive computational analysis. More precisely, for each $4\leq n\leq 10$ and each relevant value of $R$, we examined all $2^n$ admissible initial states and iterated the dynamics until either the all-healthy state was reached or a previously visited state was repeated.

For $11\leq n\leq 26$, we carried out additional computational experiments. The values reported in Table~\ref{tab:extinction-sets-wheels} for this range should therefore be regarded as experimental evidence rather than as an exhaustive determination of the extinction sets. These computations suggest the parity-dependent behavior stated in Conjecture~\ref{conj:wheel-extinction-sets}.

\begin{table}[ht]
\centering
\begin{tabular}{c|c|c|c}
$n$ & $\mathcal{E}(W_n)$ & $\hiv(W_n)$ & $\HIV(W_n)$\\
\hline
$4$  & $\{R\in\mathbb{Z^+}:R\geq3\}$           & $3$  & $3$\\
$5$  & $\{R\in\mathbb{Z^+}:R\geq4\}$           & $4$  & $4$\\
$6$  & $\{R\in\mathbb{Z^+}:R\geq5\}$           & $5$  & $5$\\
$7$  & $\{R\in\mathbb{Z^+}:R\geq6\}$           & $6$  & $6$\\
$8$  & $\{R\in\mathbb{Z^+}:R\geq7\}$           & $7$  & $7$\\
$9$  & $\{R\in\mathbb{Z^+}:R\geq8\}$           & $8$  & $8$\\
$10$ & $\{R\in\mathbb{Z^+}:R\geq9\}$           & $9$  & $9$\\
$11$ & $\{R\in\mathbb{Z^+}:R\geq10\}$          & $10$ & $10$\\
$12$ & $\{3\}\cup\{R\in\mathbb{Z^+}:R\geq11\}$ & $3$  & $11$\\
$13$ & $\{R\in\mathbb{Z^+}:R\geq12\}$          & $12$ & $12$\\
$14$ & $\{3\}\cup\{R\in\mathbb{Z^+}:R\geq13\}$ & $3$  & $13$\\
$15$ & $\{R\in\mathbb{Z^+}:R\geq14\}$          & $14$ & $14$\\
$16$ & $\{3\}\cup\{R\in\mathbb{Z^+}:R\geq15\}$ & $3$  & $15$\\
$17$ & $\{4\}\cup\{R\in\mathbb{Z^+}:R\geq16\}$ & $4$  & $16$\\
$18$ & $\{3\}\cup\{R\in\mathbb{Z^+}:R\geq17\}$ & $3$  & $17$\\
$19$ & $\{4\}\cup\{R\in\mathbb{Z^+}:R\geq18\}$ & $4$  & $18$\\
$20$ & $\{3\}\cup\{R\in\mathbb{Z^+}:R\geq19\}$ & $3$  & $19$\\
$21$ & $\{4\}\cup\{R\in\mathbb{Z^+}:R\geq20\}$ & $4$  & $20$\\
$22$ & $\{3\}\cup\{R\in\mathbb{Z^+}:R\geq21\}$ & $3$  & $21$\\
$23$ & $\{4\}\cup\{R\in\mathbb{Z^+}:R\geq22\}$ & $4$  & $22$\\
$24$ & $\{3\}\cup\{R\in\mathbb{Z^+}:R\geq23\}$ & $3$  & $23$\\
$25$ & $\{4\}\cup\{R\in\mathbb{Z^+}:R\geq24\}$ & $4$  & $24$\\
$26$ & $\{3\}\cup\{R\in\mathbb{Z^+}:R\geq25\}$ & $3$  & $25$\\
\end{tabular}
\caption{Extinction sets and HIV parameters of the wheel graphs
$W_n$ for $4\leq n\leq26$.}
\label{tab:extinction-sets-wheels}
\end{table}

The experimental results reveal a parity-dependent pattern apart from a small number of exceptional cases. For even wheels, the first isolated replacement value that guarantees extinction for every admissible initial state appears to be $R=3$, whereas for sufficiently large odd wheels it appears to be $R=4$. In both cases, extinction is not guaranteed again until the replacement parameter reaches $n-1$. This observation motivates the following conjecture.

\begin{conjecture}\label{conj:wheel-extinction-sets}
The extinction sets of the wheel graphs satisfy $\mathcal{E}(W_n) = \{3\}\cup\{R\in\mathbb{Z^+}:R\geq n-1\}$ for every even $n\geq12$, and $\mathcal{E}(W_n) = \{4\}\cup\{R\in\mathbb{Z^+}:R\geq n-1\} $ for every odd $n\geq17$.Consequently,
\[\hiv(W_n)=\begin{cases}
3, & \text{if $n\geq12$ is even},\\
4, & \text{if $n\geq17$ is odd},
\end{cases}
\qquad\text{and}\qquad \HIV(W_n)=n-1.\]
\end{conjecture}

Although wheel graphs form a structurally simple family, the computational behavior of the extinction parameters $\HIV$ and $\hiv$ is already surprisingly rich. In particular, wheels provide a natural test case for understanding the distinction between the first replacement value that guarantees extinction and the threshold beyond which extinction occurs for every larger value. The isolated values observed in Table~\ref{tab:extinction-sets-wheels} also show that the extinction property is not necessarily monotonically with respect to the replacement parameter.

A complete proof of Conjecture~\ref{conj:wheel-extinction-sets} appears to require techniques capable of relating local transition rules with the global structure of periodic orbits. Possible tools include phase differences along cycles, propagation and annihilation of local fronts, and the structure of the corresponding state-transition digraphs. Thus, a more detailed study of wheel graphs could not only lead to a complete characterization of their extinction sets, but also provide techniques applicable to broader graph families and contribute to a more general understanding of the mechanisms that govern latency and extinction in the Mukwembi dynamics.

\section{Conclusions}\label{sec:conclusions}

In this paper we studied Mukwembi's graph-based model for HIV infection. We introduced the extinction parameters $\HIV(G)$ and $\hiv(G)$, which distinguish between the first replacement parameter $R$ that guarantees extinction for every admissible initial configuration, and the threshold of the replacement parameter above which extinction always occurs. These parameters reveal that extinction is not necessarily a monotone property with respect to the replacement parameter.

We established general bounds for both parameters and characterized the graphs attaining the extremal value $\HIV(G)=\Delta(G)+1$. We also proved that the gap between the two extinction parameters, $\HIV(G)-\hiv(G)$, is not bounded. We determined explicit values of both parameters for classical graph families, providing concrete examples of the general theory. Also, we introduced the $R$-state transition digraph, which provides a natural representation of the model as a cellular automaton. This viewpoint allows extinction and latency to be interpreted in terms of connected components and directed cycles. Within this framework, we proved that cycles of arbitrary length can occur and characterized the existence of cycles of length two.

\begin{figure}[h]
\centering
\begin{tikzpicture}[every node/.style={inner sep=0pt,	font=\small}]
    \node (0) at (0,1) {$\bullet$};
    \node (1) at (0,-0.5) {$\bullet$};
    \node (2) at (-1,-0.5) {$\bullet$};
    \node (3) at (-1,0) {$\bullet$};
    \node (4) at (-1,0.5) {$\bullet$};
    \node (5) at (-1,1) {$\bullet$};
    \node (6) at (1,-0.5) {$\bullet$};
    \node (7) at (1,1) {$\bullet$};
    \node (8) at (1.5,0.25) {$\bullet$};
    \draw (0)--(1)--(2)--(0)--(3)--(1)--(4)--(0)--(5)--(1)--(6)--(0)--(7)--(1);
    \draw (6)--(8)--(7);
\end{tikzpicture}
\caption{\label{fig:example3intervals}A graph such that  $\mathcal{E}(G)$ is conformed by three separated integer intervals}
\end{figure}
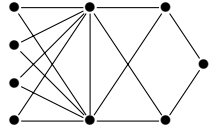

In the case that $\HIV(G)=\hiv(G)$, the set $\mathcal{E}(G)$ is the integer interval $[\HIV(G),\infty)$. We found graphs for which $\mathcal{E}(G)$ is not composed for a single interval, for example, the flag graph $F_{k}$ satisfies $\mathcal{E}(F_{k})=\{k+1\}\cup [2k+1,\infty)$. We explore this condition computationally, and we found a graph $G$ for which the set $\mathcal{E}(G)$ consist of three separated intervals, since $\mathcal{E}(G)=\{3\}\cup \{5\}\cup [7,\infty)$ (see Figure \ref{fig:example3intervals}). This graph is the only one with at most $9$ vertices where $\mathcal{E}$ is conformed by three intervals. We suspect that the number of intervals of $\mathcal{E}(G)$ is not bounded. Despite this example, we did not find examples for which $\mathcal{E}(G)$ was conformed by $4$ or more intervals.

The results of this paper suggest several directions for future research. A natural problem is to determine the parameters $\hiv(G)$ and $\HIV(G)$ for broader graph classes or to obtain sharper bounds in terms of structural graph invariants beyond the maximum and minimum degrees. Another more general direction is to understand the structure of the latency and extinction sets $\mathcal{L}(G)$ and $\mathcal{E}(G)$. The examples constructed in this paper show that the transition from latency to extinction need not be monotone as the replacement parameter increases.

From the dynamical viewpoint, it would also be interesting to investigate the structure of the $R$-state transition digraph in greater depth, including the number and size of its weakly connected components. We also characterized the graphs for which $D_R(G)$ contains cycles of length $2$, but a characterization of longer cycles remains open. Finally, the computational complexity of determining the extinction parameters and the development of efficient algorithms for their computation remain largely open problems. 

The code we used is available at \href{https://github.com/XGEu2X/GraphDynamicSpreading.}{https://github.com/XGEu2X/GraphDynamicSpreading}.

\section*{Acknowledgments}

This work was supported by the Secretaría de Ciencia, Humanidades, Tecnología e Innovación (SECIHTI) through project CBF-2025-G-1435. M. A. Espinosa-García acknowledges the financial support provided by the Secretaría de Ciencia, Humanidades, Tecnología e Innovación (SECIHTI) through the Estancias Posdoctorales por México 2025 program. G. L. Maldonado acknowledges the financial support provided by the UNAM-DGAPA Postdoctoral Fellowship Program at the Universidad Nacional Autónoma de México.

\bibliographystyle{elsarticle-num}
\bibliography{refs}

\end{document}